\documentclass{amsart}
\usepackage{amsthm}
\usepackage{lmodern}
\usepackage[english]{babel}
\usepackage{amsmath}
\usepackage{amssymb}
\usepackage{mathptmx} 
\usepackage{color}
\usepackage{graphicx}
\usepackage{caption}
\usepackage{subcaption}
\usepackage{float}
\usepackage[all,cmtip]{xy}
\usepackage{bm}
\usepackage{pinlabel}
\usepackage{enumerate}
\newtheorem{theorem}{\bf Theorem}[section]
\newtheorem{lemma}[theorem]{\bf Lemma}
\newtheorem{definition}[theorem]{\bf Definition}
\newtheorem{corollary}[theorem]{\bf Corollary}
\newtheorem{proposition}[theorem]{\bf Proposition}
\newtheorem{question}[theorem]{\bf Question}
\newtheorem{remark}[theorem]{\bf Remark}

\newcommand{\rme}{\mathrm{e}}
\newcommand{\rmi}{\mathrm{i}}
\newcommand{\rmd}{\mathrm{d}}

\newcommand{\defeq}{\mathrel{\mathop:}=}

\begin{document}

\title{Real algebraic links in $S^3$ and braid group actions on the set of $n$-adic integers}

\author{
Benjamin Bode}
\date{}

\address{Department of Mathematics, Osaka University, Toyonaka, Osaka 560-0043, Japan}
\email{ben.bode.2013@my.bristol.ac.uk}




\maketitle
\begin{abstract}
We construct an infinite tower of covering spaces over the configuration space of $n-1$ distinct non-zero points in the complex plane. This results in an action of the braid group $\mathbb{B}_n$ on the set of $n$-adic integers $\mathbb{Z}_n$ for all natural numbers $n\geq 2$. We study some of the properties of these actions such as continuity and transitivity. The construction of the actions involves a new way of associating to any braid $B$ an infinite sequence of braids, whose braid types are invariants of $B$. We present computations for the cases of $n=2$ and $n=3$ and use these to show that an infinite family of braids close to real algebraic links, i.e., links of isolated singularities of real polynomials $\mathbb{R}^4\to\mathbb{R}^2$.  
\end{abstract}
%

\section{Introduction}\label{sec:intro}
Actions of a group on a set can offer insights both on properties of the group and symmetries of the set. In the case of the Artin braid group on $n$ strands $\mathbb{B}_n$ there is an additional motivation, since the elements of the group correspond to topological objects. One might hope that a braid group action contains information not only about the group structure, but also about the topological interpretation of its elements, say about the link type of their closures for example. By Markov's, Theorem elements of the braid group correspond to braids that close to equivalent links if and only if they are related through a sequence of  isotopies, conjugation and (de)stabilisation moves. Actions that contain information that does not change under conjugation and (de)stabilisation can therefore be used to construct link invariants. 

It should be noted that actions that are actually representations particularly lend themselves to this procedure as they offer several values that are invariant under conjugation, such as the determinant or the trace of the matrices.
This principle has been applied many times, such as in the case of the interpretation of the Alexander polynomial as a normalised determinant of the Burau representation \cite{burau} or the Jones polynomial as a normalised Markov trace of a representation of the braid group into a Temperley-Lieb Algebra \cite{jones}.

In this paper we are going to construct braid group actions on the $n$-adic integers $\mathbb{Z}_n$. While these do have an interesting algebraic structure, as well as being a prominent object of study in itself in number theory, we do not expect these actions to lead to representations of the braid group, since the actions do not respect the algebraic structure of $\mathbb{Z}_n$ and treat it as a set. However, they offer interesting connections between different aspects of the study of braid groups, linking group theoretic properties with the topology of certain configuration spaces and subsets of the space of complex polynomials.

Let $n$ be a natural number with $n\geq2$. Let $C_n$ be the space of monic complex polynomials $f\in \mathbb{C}[z]$ of degree $n$ with $n$ distinct roots, or equivalently, the configuration space of $n$ distinct unmarked points in the complex plane. Then $\pi_1(C_n)=\mathbb{B}_n$, where $\mathbb{B}_n$ denotes the braid group on $n$ strands. This article is to a large part about the consequences of a result by Beardon, Carne and Ng \cite{critical}, namely about the subset $Z_n\subset C_n$ of polynomials $f\in Z_n$ that have $n-1$ distinct non-zero critical values and a constant term equal to 0. Then the set of possible sets of critical values of such a polynomial is
\begin{equation}
V_n\defeq\{(v_1,v_2,\ldots,v_{n-1})\in(\mathbb{C}\backslash\{0\})^{n-1}: v_i\neq v_j \text{ if }i\neq j\}/S_{n-1},
\end{equation}
where $S_{n-1}$ is the symmetric group on $n-1$ elements.

\begin{theorem}
\label{thm:cover}
(Beardon-Carne-Ng \cite{critical}) The map $\theta_n:Z_n\to V_n$ that sends a polynomial $f\in Z_n$ to the set of its critical values $(v_1,v_2,\ldots,v_{n-1})$ is a covering map of degree $n^{n-1}$.
\end{theorem} 

We use this to prove the following result.

\begin{theorem}
\label{thm:fibspace}
There is a tower of covering spaces
\begin{equation}
\label{eq:tower}
\ldots\to Z_n^{i+1}\to Z_n^{i}\to\ldots \to Z_n^{2}\to Z_n^{1}=Z_n\to V_n\simeq D_n\xrightarrow{p} C_n,
\end{equation} 
where $p$ is a covering map of degree $n$, all other arrows are covering maps of degree $n^{n-1}$ and $\simeq$ denotes homotopy equivalence. 
\end{theorem}

The fiber over a point $v\in V_n$ is the set of $n^{n-1}$-adic integers $\mathbb{Z}_{n^{n-1}}$.
Since $\pi_1(C_n)=\mathbb{B}_n$, we obtain a braid group action on the set $\mathbb{Z}/n\mathbb{Z}\times\mathbb{Z}_{n^{n-1}}$, which is homeomorphic to the $n$-adic integers $\mathbb{Z}_n$, via monodromy. We write this action as $\varphi_{n}(\cdot,B):\mathbb{Z}_n\to\mathbb{Z}_n$, $B\in \mathbb{B}_n$.

With very similar considerations we construct another action $\psi_n$ of the braid group on $\mathbb{Z}_{n^n}\cong\mathbb{Z}_n$ that shares some properties with $\phi_n$, but makes computations significantly simpler. 
The main properties of the constructed actions are summarised in the following theorem.


\begin{theorem}
\label{thm:main}
For both of the constructed actions $\phi_n$ and $\psi_n$ the following are true.
\begin{enumerate}[i)]
\item They preserve the metrics on $\mathbb{Z}/n\mathbb{Z}\times\mathbb{Z}_{n^{n-1}}$ and $\mathbb{Z}_{n^n}$, respectively. Therefore they yield continuous braid group actions on $\mathbb{Z}_n$.
\item They correspond to sequences of homomorphisms $\mathbb{B}_n\to S_{n\times(n^{n-1})^j}$ and $\mathbb{B}_n\to S_{(n^n)^j}$, respectively. For $\phi_n$ the resulting action on $n\times(n^{n-1})^j$ points is transitive for all $n$ and $j$.
\item The kernels $N_j$ and $H_j$ of the homomorphisms $\mathbb{B}_n\to S_{n\times(n^{n-1})^j}$ and $\mathbb{B}_n\to S_{(n^n)^j}$, respectively, form descending series of normal subgroups of the braid group that do not stabilise.
\end{enumerate}
\end{theorem}

Regarding the last point we would like to point out that a proof of the faithfulness of the actions would be desirable, which is equivalent to the intersections $\bigcap_j N_j$ and $\bigcap_j H_j$, respectively, being trivial. 


The remainder of this paper is structured as follows. In Section \ref{sec:background} we provide the reader with the necessary background on the connection between polynomial maps and braids as well as the necessary definitions and basic results from the theory of $n$-adic integers. In Section \ref{sec:actions} we construct the actions $\phi_n$ and $\psi_n$ and prove Theorem \ref{thm:fibspace} and Theorem\ref{thm:main} \textit{i)}.

The definition of $\psi_n$ involves the construction of a sequence of braids 
\begin{equation}
(B,\{B_{1,0},B_{1,1},\ldots,B_{1,n^n-1}\},\{B_{2,0},B_{2,1},\ldots,B_{2,(n^n)^2-1}\},\ldots,\{B_{j,0},B_{j,1},\ldots,B_{j,(n^n)^j-1}\},\ldots),
\end{equation}
which is an invariant of the braid $B$. We outline in Section \ref{sec:seq} how sequences like this can be used to improve invariants of braids or conjugacy classes, i.e., make the invariants better at distinguishing different braids (or different conjugacy classes of braids).
In Section \ref{sec:comps} we compute the effect of $\psi_n(\cdot,\sigma_i)$ on the first coordinates of $\mathbb{Z}_{n^n}$ for the generators $\sigma_i$ of the braid groups on two and three strands. The results can be used to compute the effect of $\psi_n(\cdot,B)$ on any fixed number of coordinates for any braid $B$ on two or three strands in a computational complexity that is only linear in the length of $B$. 
In Section \ref{sec:dynamics} we prove Theorem \ref{thm:main} \textit{ii)} and take a brief look at the orbit of points under the repeated application of $\theta_n$ and $\theta_n^{-1}$ for $n=2,3$.
 Section \ref{sec:normal} studies the sequences of the normal subgroups that are given by the kernels of the in the statement of Theorem \ref{thm:main} \textit{ii)} and we show that the sequences of normal subgroups do not stabilise, which proves Theorem \ref{thm:main} \textit{iii)}.
In Section \ref{sec:realalg} we employ the computations from Section \ref{sec:comps} to show that an infinite family of braids close to real algebraic links, i.e., links of isolated singularities of real polynomials $\mathbb{R}^4\to\mathbb{R}^2$. We obtain the following result. 
\begin{theorem}
\label{thm:real}
Let $\epsilon\in\{\pm 1\}$ and let $B=\prod_{j=1}^{\ell}w_{i_j}^{\epsilon}$ be a 3-strand braid with
\begin{align}
&{\makebox[3.5cm][l]{$w_1=\sigma_2$,}}{\makebox[3cm][l]{$w_2=\sigma_1^2$,}}{\makebox[2.5cm][l]{$w_3=(\sigma_1\sigma_2\sigma_1)^2$,}}\nonumber\\
&{\makebox[3.5cm][l]{$w_4=(\sigma_2\sigma_1\sigma_2^{-1}\sigma_1\sigma_2)^2$,}}{\makebox[2cm][l]{$w_5=\sigma_2^{-1}\sigma_1\sigma_2\sigma_2\sigma_1$,}}
\end{align}
and such that there is a $j$ such that $i_j=3$ or such that there is only one residue class $k \text{ mod }3$ such that $i_j\neq k$ for all $j=1,2,\ldots,\ell$.
Then the closure of $B^2$ is real algebraic.
\end{theorem}

We show that this family contains links that have not been known to be real algebraic previously.

\textbf{Acknowledgements:} The author is grateful to Andreas Bode, Tara Brendle, Mark Dennis, Daniel Peralta-Salas, Jonathan Robbins, Dale Rolfsen and Osamu Saeki for valuable discussions and encouragements. The author would also like to express his gratitude towards Seiichi Kamada for comments on earlier drafts of the paper and to the referees, whose comments have significantly improved the paper. The author was supported by the Leverhulme Trust Research Programme Grant RP2013-K-009, SPOCK: Scientific Properties Of Complex Knots and by JSPS KAKENHI Grant Number JP18F18751 and
a JSPS Postdoctoral Fellowship as JSPS International Research Fellow.
Some of the computations in Section 5 were performed during a research visit at the ICMAT in Madrid, which was made possible by the generosity of Daniel Peralta-Salas and the ERC grant 335079.

\section{Background}
\label{sec:background}
In this section we summarise the necessary background on $n$-adic integers, polynomials and their relation to braids, as well as Beardon, Carne and Ng's covering map from \cite{critical}. Proofs and more detailed descriptions can be found in \cite{critical}, \cite{bode:2016lemniscate}, \cite{bode:2016polynomial} and \cite{bode:thesis}, while the basics of $n$-adic integers are available in countless number theory texts such as \cite{adicbook} or \cite{neukirch}.

\subsection{The $n$-adic integers}
\label{sec:adic}
We only need the very basics of the theory of $n$-adic integers $\mathbb{Z}_{n}$. They are defined as the inverse limit $\underset{j}{\varprojlim}\mathbb{Z}/n^{j}\mathbb{Z}$ for $j\geq 1$. Therefore it is the set of infinite sequences
\begin{equation}
\label{eq:adic}
a=(a_1,a_2,a_3,\ldots)
\end{equation}
such that $a_j\in\mathbb{Z}/n^{j}\mathbb{Z}$ and  $a_{j+1}=a_j\text{ mod }n^{j}$ for all $j\geq 1$.

We can put an $n$-adic valuation (and with it a metric and a topology) on $\mathbb{Z}_{n}$ by defining
\begin{equation}
ord_{n}(a)=\min\{k\geq 1:a_j= 0\text{ for all }j< k\}.
\end{equation}
The order of the 0-sequence $ord_n(0)$ is defined as $\infty$. 
The $n$-adic value of $a=(a_1,a_2,a_3,\ldots)$ is
\begin{equation}
|a|_{n}=n^{-ord_{n}(a)}.
\end{equation} 
This means an $n$-adic integer is considered small if its sequence starts with a lot of zeros. Consequentially, two $n$-adic integers are considered close to each other if their sequences start with many identical terms. The topology that is induced by this metric is simply the profinite topology.



If $n=\prod_{i=1}^{m}p_i^{k_i}$ with $p_i$ prime, $p_i\neq p_j$ for all $i\neq j$ and $k_i\geq 1$, then we have a homeomorphism
\begin{equation}
\label{eq:iso}
\mathbb{Z}_{n}\cong\mathbb{Z}_{\prod_{i=1}^{m}p_i}\cong \mathbb{Z}_{p_1}\times\mathbb{Z}_{p_2}\times\ldots\times\mathbb{Z}_{p_m}.
\end{equation}

The set of the $n$-adic integers is also homeomorphic to the Cantor set $\Omega$ for all $n$. Since the actions in Theorem \ref{thm:main} are actions on the set of $n$-adic integers and do not respect the algebraic structure of that set, we could regard them as actions on $\Omega$ instead. In the light of Theorem \ref{thm:main}i) it will be more appropriate to regard them as actions on the metric spaces $\mathbb{Z}/n\mathbb{Z}\times\mathbb{Z}_{n^{n-1}}$ and $\mathbb{Z}_{n^n}$, respectively.



\subsection{Polynomials, braids and the covering map}
We discuss different ways in which a loop in a space of polynomials can be interpreted as a braid.

\textbf{Example:} Let $C_n$ be the space of monic complex polynomials in one variable of degree $n$ and with $n$ distinct roots.
Consider a parametrised family of polynomials $f_t\in C_n$, $t\in[0,1]$. Alternatively, this can be written as a map $f:\mathbb{C}\times[0,1]\to\mathbb{C}$, $(z,t)\mapsto f_t(z)$, which is a polynomial in $z$. This is a Weierstrass polynomial as discussed in the context of braids in \cite{hansen}, \cite{hansenbook} and \cite{moller}. 
Since $f_t$ has $n=n_t=\deg f_t$ distinct roots for every $t\in[0,1]$, the nodal set $f^{-1}(0)$ forms a braid on $n$ strands in $\mathbb{C}\times[0,1]$. The underlying principle is the fundamental theorem of algebra that allows us to identify a monic polynomial with distinct roots with its (unordered) set of roots. The map that sends a polynomial in $C_n$ to its roots gives a homeomorphism $C_n\cong \{(z_1,z_2,\ldots,z_n)\in\mathbb{C}^n:z_i\neq z_j \text{ if }i\neq j\}/S_{n}$, where $S_{n}$ is the symmetric group on $n$ elements. This means that the fundamental group of the space of monic complex polynomials of a fixed degree $n$ with distinct roots is $\mathbb{B}_n$, with homotopies of loops in $C_n$ corresponding to braid isotopies. For a path in $C_n$ given by $f_t$, $t\in[0,1]$, we will frequently refer to the braid that is formed by the roots of $f_t$, $t\in[0,1]$ as the braid corresponding to the path $f_t$.\\

This is of course well-known, but what is less often considered is the possibility of replacing the concept of `roots' in the above construction by other sets such as the critical points or critical values of the polynomials.\\

\textbf{Example:} Instead of demanding that each $f_t$ has $n$ distinct roots, we can require that $f_t$ has $n-1$ distinct critical points $c_1,c_2,\ldots,c_{n-1}$ with $f_t'(c_i)=\tfrac{\partial f_t}{\partial z}(c_i)=0$ for all $i$. Just like above we get a homeomorphism between a space of polynomials (monic with fixed degree $n$, distinct critical points and with constant term equal to 0) and the configuration space of $n-1$ distinct complex numbers, which tells us that the fundamental group of that space of polynomials is the braid group $\mathbb{B}_{n-1}$. Restricting the space of polynomials to those with a constant term equal to zero was necessary to get a map that is 1-to-1. We might argue that this is not really different from the construction above. We merely used the fact that every polynomial has a unique antiderivative once the integration constant is fixed (i.e., the constant term is set equal to 0) and then employed the homeomorphism between $C_{n-1}$ and the corresponding configuration space.\\

It is therefore easy to construct parametrised families of polynomials $f_t\in\mathbb{C}[z]$, $t\in [0,1]$ whose nodal sets or whose sets of critical points form a given braid. In the first case, we only have to find a parametrisation 
\begin{equation}
\bigcup_{j=1}^s (z_j(t),t) \subset \mathbb{C}\times [0,1]
\end{equation} 
of the braid on $s$ strands and define $f_t(z)=\prod_{j=1}^s (z-z_j(t))$. Here the polynomial degree $n$ equals the number of strands $s$ and the roots trace out the desired braid as $t$ varies from 0 to 1.

In the latter case, we obtain $f_t$ via
\begin{equation}
f_t(z)=\int_0^z \prod_j^{s} (w-z_j(t)) \rmd w. 
\end{equation} 
Note that here each $f_t$ has degree $n=s+1$ and the critical points trace out the desired braid as $t$ varies from 0 to 1.\\

\textbf{Example:} Suppose now we are not interested in the topology of the nodal set or that of the critical set of a family of polynomials, but instead in the topology of the set of critical values, i.e., the values $f_t(c_i)$ of the polynomials at their critical points $c_i$. Similarly to before, if we demand that the critical values $(v_1,v_2,\ldots,v_{n-1})=(f(c_1),f(c_2),\ldots,f(c_{n-1}))$ of a monic polynomial (of degree $n$ and with constant term equal to zero) are disjoint, the map that sends such a polynomial to its set of critical values allows us to interpret loops in the space of such polynomials as braids. However, this map is not a homeomorphism and the problem of constructing polynomials for a given braid parametrisation becomes more challenging. 

Given a parametrisation
\begin{equation}
\bigcup_{j=1}^s (v_j(t),t) \subset \mathbb{C}\times [0,1]
\end{equation}
of a braid $B$ we want to construct a $f(z,t)=f_t(z)$ such that the critical values of $f_t$ form the braid $B$, i.e., we want the existence of $c_1(t),c_2(t),\ldots,c_s(t)$ such that $f(c_j(t),t)=v_j(t)$ and $f'(c_j(t),t)=0$ for all $j=1,2,\ldots,s$ and all $t$. 

We find such a family of polynomials $f_t$ for a given braid parametrisation by solving a system of polynomial equations for every $t\in[0,1]$, which is not very practical. The problem becomes easier if we are content with a family of polynomials $f_t$ whose critical values form a braid that is isotopic to $B$, rather than realising a specific parametrisation. There is an extra degree of freedom that can be eliminated by setting the constant term of $f_t$ to 0 for all $t$. In this case, we can solve the system of polynomial equations for some fixed values of $t$, say $t=t_1,t_2,\ldots,t_m$ for some $m\in\mathbb{N}$, to obtain polynomials $f_{t_i}$. In contrast to the earlier examples, these solutions are not unique. We return to the question of the number of solutions later. Interpolating functions through the coefficients of the polynomials $f_{t_i}$ then provide us with the coefficients of $f(z,t)=f_t(z)$ as functions of $t$. For a sufficiently large number of data points, i.e., high values of $m$, the braid that is formed by the critical values of $f$ is isotopic to $B$.\\

This brings us to the question that started this project initially. Suppose the polynomials $f_t$ whose critical values form the braid $B$ all lie in $C_n$ with $n=s+1$. Then the roots of $f_t$ form a braid too, say $A$. What can be said about the relation between the braids $A$ and $B$, one formed by the roots of $f_t$, the other by its critical values? 

Let
\begin{align}
Z_n\defeq\{&f\in\mathbb{C}[z]: f \text{ monic of degree }n \text{ with distinct roots,}\nonumber\\
&  \text{distinct critical values and constant term equal to }0\}.
\end{align} 

\begin{remark} Since the critical values are distinct, the critical points of any polynomial in $Z_n$ are also distinct. The space $Z_n$ is therefore precisely the intersection of the three spaces considered earlier, the space of polynomials with distinct roots, distinct critical points and critical values respectively.
\end{remark}

If $f_t\in Z_n\subset C_n$, the fact that its roots are distinct implies that none of the critical values $v_1(t), v_2(t),\ldots, v_{n-1}(t)$ is 0. The space of the possible sets of critical values is therefore
\begin{equation}
V_n\defeq\{(v_1,v_2,\ldots,v_{n-1})\in(\mathbb{C}\backslash\{0\})^{n-1}: v_i\neq v_j \text{ if }i\neq j\}/S_{n-1},
\end{equation}
The fundamental group of $V_n$ is the affine braid group $\mathbb{B}_{n-1}^{aff}$, which is defined as follows.
\begin{definition}
The affine braid group $B_{n-1}^{aff}$ is presented by the generators $x,\sigma_2,\sigma_3,\ldots,\sigma_{n-1}$ and the relations
\begin{align}
\sigma_2x\sigma_2x&=x\sigma_2x\sigma_2,&\nonumber\\
\sigma_i\sigma_j&=\sigma_j\sigma_i, &\quad\text{ if }|i-j|>1,\nonumber\\
\sigma_i\sigma_{i+1}\sigma&=\sigma_{i+1}\sigma_i\sigma_{i+1}, &\quad\text{ if }i=2,3,\ldots,n-2,\nonumber\\
\sigma_ix&=x\sigma,&\quad\text{ if }i>2.
\end{align}
\end{definition}
This corresponds to the usual Artin braid group with the only difference that in any braid word $\sigma_1$ can appear only appear with even exponents, so that $x=\sigma_1^2$ becomes one of the generators. In other words the affine braid group consists of precisely those braid words, where $\sigma_1$ always comes with an even exponent. Geometrically, this means that the braid can be parametrised such that the first strand does not move at all, i.e., can be taken to be $(0,t)\subset\mathbb{C}\times[0,1]$. We often refer to this strand as the \textit{0-strand} or the \textit{flagpole}. 

This description of $\pi_1(V_n)$ as $\mathbb{B}_{n-1}^{aff}$ with the above generators and relations assumes that we have chosen a base point in $V_n$ that is a $n-1$-tuple of complex numbers, whose real parts are all positive. If we consider different base points, loops correspond to conjugates (in the braid group) of affine braids with the generators and relations above, but setting $x=\sigma_1^2$. In general, any loop in $V_n$ corresponds to a braid, whose permutation of the $n$ strands has (at least) one fixed point. Sometimes we will be somewhat inexact and call such a braid affine as well. It should be clear from the context if we refer to a braid, whose permutation fixes the first strand or the $i^{\text{th}}$ strand for some $i=2,3,\ldots,n$, i.e., which of the strands is the 0-strand. 

The fact that loops in $V_n$ are affine braids is a first hint that the braid word of $B$, which is formed by the critical values, does not carry all relevant information about the parametrisation $(v_1(t),v_2(t),\ldots,v_{n-1}(t))$. The way in which the strands twist around the flagpole $(0,t)\subset\mathbb{C}\times[0,1]$ is important. We should focus on the braid that is formed by $(0,v_1(t),v_2(t),\ldots,v_{n-1}(t))$ instead of $B$ itself.

The question that was raised above should therefore be changed to the following.
\begin{question}
Given a loop $f_t\subset Z_n$, $t\in[0,1]$, $f_0=f_1$, with roots $(z_1(t),z_2(t),\ldots,z_n(t))$ and critical values $(v_1(t),v_2(t)\ldots,v_{n-1}(t))$. What is the relation between the braid $A$ that is formed by the roots $(z_1(t),z_2(t),\ldots,z_n(t))$ and the braid $B'$ that is formed by the 0-strand and the critical values $(0,v_1(t),v_2(t)\ldots,v_{n-1}(t))$? 
\end{question}

Note that all polynomials in $Z_n$ have 0 as one of their roots, since their constant term is by definition equal to 0. Therefore the braid $A$ that is formed by the roots of a parametrised family of polynomials $f_t$ in $Z_n$ contains one strand that does not move at all. In other words, if $f_0=f_1$, $A$ is an affine braid.

Furthermore, since the critical values are distinct, so are the critical points of any polynomial in $Z_n$. They are also non-zero because the roots are distinct and one of the roots is 0. Thus the union of critical points of a parametrised family of polynomials $f_t\in Z_n$, $t\in[0,1]$ and the 0-strand $(0,t)\subset\mathbb{C}\times[0,1]$ form a braid too. We can therefore not only consider relations between the braids formed by the roots and the critical values of such a family of polynomials, but also their relation to the braid that is formed by the critical points and the 0-strand. 




\section{Braid group actions on $\mathbb{Z}_n$}
\label{sec:actions}
In this section we describe actions of the Artin braid group on the $n$-adic integers and show that they are continuous. In fact, they are isometries on spaces that are homeomorphic to the $n$-adic integers.

\subsection{The proof of Theorem \ref{thm:fibspace}}
By Theorem \ref{thm:cover} $Z_n$ is a covering space of $V_n$ of degree $n^{n-1}$. Therefore, the fundamental group of $V_n$, which is the affine braid group $\mathbb{B}_{n-1}^{aff}$, acts on the fibre consisting of $n^{n-1}$ points. Furthermore, recall that the polynomials in $Z_n$ have a constant term equal to 0 and therefore one of their roots is equal to $0$. Hence $Z_n$ can be embedded in $V_n$ by sending a polynomial $f\in Z_n$ to its $n-1$ distinct non-zero roots. 

Let $Z_n^1\defeq Z_n$ and define $Z_n^{j+1}$ to be the space of monic polynomials $f\in Z_n$ with (by definition distinct) critical values $(v_1,v_2,\ldots,v_{n-1})$ such that the polynomial $z\prod_{i}^{n-1}(z-v_i)$ is in $Z_n^{j}$. We obtain an infinite tower of covering spaces
\begin{equation}
\label{eq:tower2}
\ldots\to Z_n^{j+1}\to Z_n^{j}\to\ldots \to Z_n^{2}\to Z_n^{1}=Z_n\to V_n,
\end{equation} 
where each covering map is $\theta_n$, the map that sends a polynomial to its set of critical values, restricted to the relevant $Z_n^{j}$ and composed with the map that sends a set of critical values $(v_1,v_2,\ldots,v_{n-1})$ to the polynomial $z\prod_{i}^{n-1}(z-v_i)$. We often ignore the second map and move freely between interpreting a point in $Z_n^j$ as a $(n-1)$-tuple of complex numbers and as the corresponding polynomial and simply write $\theta_n$ for the covering map. It follows that $Z_n^{j}$ is a covering space of $V_n$ of degree $(n^{n-1})^j$.

Suppose for example $v=(v_1,v_2,\ldots,v_{n-1})\in V_n$ and let $f_j\in Z_n$, $j=1,2,\ldots,n^{n-1}$ be the $n^{n-1}$ preimages of $v$ under the covering map, i.e., the polynomials that have the set $\{v_1,v_2,\ldots,v_{n-1}\}$ as their set of critical values. Let $\{z_1^j,z_2^j,\ldots,z_{n-1}^j\}$ denote the $n-1$ non-zero roots of $f_j$. Since $(z_1^j,z_2^j,\ldots,z_{n-1}^j)\in V_n$, there are exactly $n^{n-1}$ polynomials in $Z_n$ that have $(z_1^j,z_2^j,\ldots,z_{n-1}^j)$ as their sets of critical values.
In summary, for a given $v=(v_1,v_2,\ldots,v_{n-1})\in V_n$  there are  $(n^{n-1})^2$ polynomials $g_k$ in $Z_n$ that have the property that their critical values $(z_1,z_2,\ldots,z_{n-1})$ define a polynomial $z\prod_{i}^{n-1}(z-z_i)$ that has 0 and the critical values of $g_k$ as its roots and $v$ as its set of critical values.

Let
\begin{equation}
D_n=\{(z_1,z_2\ldots,z_n)\in\mathbb{C}^n:\ z_i\neq z_j\text{ if }i\neq j\}/S_{n-1},
\end{equation}
where the action by the symmetric group $S_{n-1}$ permutes the last $n-1$ factors. The map $p:D_n\to C_n$ that sends $(z_1,z_2,\ldots,z_n)$ to $(z_1,z_2,\ldots,z_n)$ is a covering map of degree $n$. This and the following lemma was pointed to us by a referee.

\begin{lemma}
The spaces $D_n$ and $V_n$ are homotopy equivalent.
\end{lemma}
\begin{proof}Consider the maps $g:V_n\to D_n$ and $h:D_n\to V_n$ given by 
\begin{equation}
g(z_1,z_2,\ldots,z_{n-1})=(0,z_1,\ldots,z_{n-1})
\end{equation}
and
\begin{equation}
h(z_1,z_2,\ldots,z_n)=(z_2-z_1,z_3-z_1,\ldots,z_n-z_1).
\end{equation}
Then $hg$ is the identity on $V_n$ and $hg\simeq id_{D_n}$ via the homotopy
\begin{equation}
H_t(z_1,z_2,\ldots,z_n)=(tz_1,z_2-z_1+tz_1,z_3-z_1+tz_1,\ldots,z_n-z_1+tz_1).
\end{equation}
\end{proof}

Therefore $\pi_1(D_n)=\pi_1(V_n)=\mathbb{B}_{n-1}^{aff}$ and we extend the tower of covering spaces in Eq. (\ref{eq:tower2}) as follows:
\begin{equation}
\label{eq:tower3}
\ldots\to Z_n^{j+1}\to Z_n^{j}\to\ldots \to Z_n^{2}\to Z_n^{1}=Z_n\to V_n\simeq D_n\xrightarrow{p} C_n.
\end{equation}

Hence $\pi_1(C_n)=\mathbb{B}_n$ acts on the fibre in each $Z_n^j$ and in $V_n$. We thus have an action of $\mathbb{B}_n$ on the set of $n\times(n^{n-1})^j$ points, or equivalently, on the set $\mathbb{Z}/n\mathbb{Z}\times\mathbb{Z}/(n^{n-1})^j\mathbb{Z}$ for every $j\in\mathbb{Z}_{\geq 1}$. Furthermore, the different actions are compatible with each other in the sense that $\theta_n(x.\gamma)=\theta_n(x).\gamma$ for all $x$ in the fibre in $Z_n^j$ and all $\gamma\in\mathbb{B}_n$.


Let $v\in V_n$ and consider the set $X$ of infinite sequences $(a_1,a_2,a_3,\ldots)$ with $a_j\in Z_n^{j}$, $\theta_n(a_1)=v$ and $\theta_n(a_{j+1})=a_j$ for all $j\geq 1$. Since there are exactly $(n^{n-1})^j$ choices for the $j^{\text{th}}$ term $a_j$, of which only $n^{n-1}$ satisfy the compatibility condition, this set can be identified with the $n^{n-1}$-adic integers $\mathbb{Z}_{n^{n-1}}$. The fibre in the infite tower of covering spaces in Eq. (\ref{eq:tower2}) over a point $v$ is therefore $\mathbb{Z}_{n^{n-1}}$.

A given loop $(z_1(t),z_2(t),\ldots,z_n(t))$ in $C_n$ lifts to $n$ paths in $D_n$, which can be identified via the homotopy equivalence with the $n$ paths $(z_1(t)-z_i(t),z_2(t)-z_i(t),\ldots,z_{i-1}(t)-z_i(t),z_{i+1}(t)-z_i(t),\ldots,z_n(t)-z_i(t))$, $i=1,2,\ldots,n$. These $n$ paths permute the points in the fibre over the base points $(z_1(0)-z_i(0),z_2(0)-z_i(0),\ldots,z_{i-1}(0)-z_i(0),z_{i+1}(0)-z_i(0),\ldots,z_n(0)-z_i(0))$ via monodromy, so that we obtain a braid group action on the set $\mathbb{Z}/n\mathbb{Z}\times\mathbb{Z}_{n^{n-1}}$, which is homeomorphic to $\mathbb{Z}_n$. 

This finishes the proof of Theorem \ref{thm:fibspace} and it follows immediately that $\mathbb{B}_n$ acts on $\mathbb{Z}_n$ via the monodromy action. We write the action as $\phi_n(\cdot,B):\mathbb{Z}_n\to\mathbb{Z}_n$, $B\in\mathbb{B}_n$.
\begin{definition}
A point $v=(v_1,v_2,\ldots,v_{n-1})\in V_n$ is said to have 0 in $i^{\text{th}}$ position if exactly $i-1$ of its entries $v_j$ have a negative real part and exactly $n-i$ have positive real part.
\end{definition}

\begin{remark}
For a base point $z=(z_1,z_2,\ldots,z_n)$ in $C_n$, the corresponding $n$ start and end points of the lifted paths in $V_n$ are given by $x_i=(z_1-z_i,z_2-z_i,\ldots,z_{i-1}-z_i,z_{i+1}-z_i,\ldots,z_n-z_i)$. If all the $z_i$ have distinct non-zero real parts, the point $x_i$ has 0 in $i^{\text{th}}$ position (possibly after permuting the indices). A loop in $C_n$ based at $z$ and corresponding to the braid $B$ lifts to $n$ paths between $x_i$. Note that the entries of each of these paths together with the 0-strand are parametrisations of strands that form the same braid $B$. Therefore, the lift that starts at $x_i$ must end at the point $x_{\pi(B)}$, where $\pi:\mathbb{B}_n\to S_n$ is the permutation representation .
The monodromy action with respect to the covering map $p:D_n\to C_n$, which is the action on the first factor $\mathbb{Z}/n\mathbb{Z}$ inside $\phi_n$, is thus precisely the permutation representation.
\end{remark}

\subsection{The proof of Theorem \ref{thm:main}$i)$ for the action $\phi_n$}
We proceed with the proof of Theorem \ref{thm:main}$i)$.

\begin{proposition}
\label{prop:affine}
The braid group $\mathbb{B}_n$ action defined by $\phi_n$ preserves the metric on $\mathbb{Z}/n\mathbb{Z}\times\mathbb{Z}_{n^{n-1}}$. Therefore, the action $\phi_{n}(\cdot,B):\mathbb{Z}_n\to\mathbb{Z}_n$ is continuous.
\end{proposition}
\begin{proof}
The topology on $\mathbb{Z}_{n^{n-1}}$ is induced by the metric, which itself is derived from the $n^{n-1}$-adic valuation. We can easily extend the metric to $\mathbb{Z}/n\mathbb{Z}\times\mathbb{Z}_{n^{n-1}}$. 
Let $x=(x_1,x_2,x_3,\ldots)\in\mathbb{Z}/n\mathbb{Z}\times\mathbb{Z}_{n^{n-1}}$ and $y=(y_1,y_2,y_3,\ldots)\in\mathbb{Z}/n\mathbb{Z}\times\mathbb{Z}_{n^{n-1}}$, i.e., $x_i, y_i\in\mathbb{Z}/n\times (n^{n-1})^i\mathbb{Z}$ satisfy the compatibility conditions $x_i\equiv x_{i-1}\text{ mod }n\times(n^{n-1})^i$, $y_i\equiv y_{i-1}\text{ mod }n\times(n^{n-1})^i$.
Suppose that $|y-x|_n=(n^{n-1})^{-m-1}$. This is equivalent to $x$ and $y$ agreeing on the first $m$ terms of the sequence, i.e., $x_i=y_i$ for all $i=1,2,\ldots,m$, but $x_{m+1}\neq y_{m+1}$.

Therefore, the first $m$ terms of the sequences $x.\gamma$ and $y.\gamma$ also agree with each other for all $\gamma\in \mathbb{B}_n$, which means the distance between $\phi_n(x,\gamma)$ and $\phi_n(y,\gamma)$ is at most $(n^{n-1})^{-m-1}$. On the other hand, applying $\gamma^{-1}$ to $\phi_n(x,\gamma)$ and $\phi_n(y,\gamma)$ shows that their distance is at least $(n^{n-1})^{-m-1}$. 
Hence $|\phi_n(y,\gamma)-\phi_n(x,\gamma)|_n=(n^{n-1})^{-m-1}$ and the action $\phi_n$ preserves the metric on $\mathbb{Z}/n\mathbb{Z}\times\mathbb{Z}_{n^{n-1}}$. The $\epsilon-\delta$ criterion then tells us that the action is continuous on $\mathbb{Z}/n\mathbb{Z}\times\mathbb{Z}_{n^{n-1}}$ and by Section \ref{sec:adic} continuous on $\mathbb{Z}_n$.
\end{proof}

\subsection{The action $\psi_n$}
Explicit calculations of $\phi_n$ are rather elaborate. In order to compute the action of the generators $\sigma_i$ of $\mathbb{B}_n$ on $\mathbb{Z}_n$, we need to find a parametrisations of each $\sigma_i$, say $(z_1(t),z_2(t),\ldots,z_n(t))$, $t\in[0,1]$. The $n$ lifts of these loops in $V_n$ are then $(z_1(t)-z_i(t),z_2(t)-z_i(t),\ldots,z_{i-1}(t)-z_i(t),z_{i+1}(t)-z_i(t),\ldots,z_n(t)-z_(t))$, $i=1,2,\ldots,n$. We now need to lift each of these paths in $V_n$ to paths in $Z_n^j$. Every such lifting procedure corresponds to solving a system of $n-1$ polynomial equations for sufficiently many data points $\{t_i\}_{i=1,2,\ldots,m}\subset[0,1]$. Therefore, in order to compute the action of $\sigma_i$ on the first $k+1$ coordinates of $\mathbb{Z}/n\mathbb{Z}\times\mathbb{Z}_{n^{n-1}}$, $k\geq1$, we need to solve $m n (n^{n-1})^{k-1}$ systems of $2(n-1)$ polynomial equations, each of which has $n^{n-1}$ solutions. 

Once we have found the lifts, we can read off the permutations on the fibres in the different $Z_n^j$, $j=1,2,\ldots,k$, that are induced by $\sigma_i$. From these we can build the effect of any $B\in \mathbb{B}_n$ on the first $k+1$ coordinates by composing the permutations for the individual generators (and their inverses). The action on the first coordinates on $\mathbb{Z}_n$ can then be determined via the homeomorphism between $\mathbb{Z}/n\mathbb{Z}\times\mathbb{Z}_{n^{n-1}}$ and $\mathbb{Z}_n$. The composition of permutations can be done in a number of steps that grows linearly with the length of the braid word of $B$. What makes this calculation impractical is the huge number of systems of polynomial equations that we have to solve.

Since we want to avoid such unnecessarily expensive computations, we define a different action $\psi_n$ of $\mathbb{B}_n$ on $\mathbb{Z}_{n^n}\cong\mathbb{Z}_n$, which is very similar to $\phi_n$, but requires solving only $m n$ systems of $2(n-1)$ polynomials, no matter how many coordinates of $\mathbb{Z}_n$ we are interested in. As an illustration of the concept we solve the corresponding systems of equations for the cases of $n=2$ and $n=3$ in Section \ref{sec:comps} and illustrate how to use the solutions to compute the action of any braid $B\in \mathbb{B}_n$ on any given number of coordinates of $\mathbb{Z}_n$.

We now define the action $\psi_n$. Let $B\in \mathbb{B}_n$. Like in the definition of $\phi_n$ we think of $B$ as a loop in $C_n$ and lift it to $n$ paths in $V_n$. Lifting these paths to $Z_n$ results in $n\times n^{n-1}$ paths in $Z_n$, which give us a permutation of the $n\times n^{n-1}=n^n$ in the fibre that is compatible with $\pi(B)$. The action of $B$ on the first coordinate of $\mathbb{Z}_{n^n}$ via $\psi_n$ is therefore the same as its action on the second coordinate of $\mathbb{Z}/n\mathbb{Z}\times\mathbb{Z}_{n^{n-1}}$ via $\phi_n(\cdot,B)$. We are going to refer to this map more often, so we give it a name: $\sigma:\mathbb{B}_n\to S_{n^n}$.

Now comes the part that distinguishes the new action from $\phi_n$. Each of the $n\times n^{n-1}$ paths in $Z_n$ corresponds to a family of polynomials $f_{i,t}\subset Z_n$, $i=1,2,\ldots,n^n$, $t\in[0,1]$ such that the union of the critical values of $f_{i,t}$ and the zero strand $(0,t)\subset\mathbb{C}\times[0,1]$ form the braid $B$ as $t$ varies from 0 to 1. Since $f_{i,t}\in Z_n$, the roots of $f_{i,t}$ also form a braid as $t$ varies from 0 to 1, say $B_i$. Instead of lifting the $n\times n^{n-1}$ paths in $Z_n$ to $Z_n^2$ as before (i.e., using the particular parametrisations of $B_i$ that we obtained through the lifting procedure), we consider each $B_i$ as a loop in $C_n$ just like we did for the original braid $B$. We can lift these $n^n$ loops in $C_n$ to paths in $V_n$ and then $Z_n$ and obtain $n^n$ permutations of the $n^n$ points in the fibre. By definition these permutations are simply $\sigma(B_i)$.

We can write these $n^n$ permutations $\sigma(B_i)$ of $n^n$ points as one permutation of $(n^n)^2$ points that is compatible with the permutation of $n^n$ points, i.e., compatible with the action on the first coordinate of $\mathbb{Z}_{n^n}$, as follows. Label the points with 0 through $n^{2n}-1$, so that every number can be expressed uniquely as $n^n k+i$ with $k,i=0,1,2,\ldots,n^n-1$. Then the permutation of $(n^n)^2$ points is
\begin{equation}
\label{eq:perm}
n^n k+i \mapsto n^n \sigma(B_i)(k)+\sigma(B)(i).
\end{equation} 
Note that the residue class mod $n^n$ of the image is given by $\sigma(B)$, which means that this permutation on $(n^n)^2$ points is compatible with $\sigma(B)$ on $n^n$ points. In each residue class are $n^n$ elements that $n^n k+i$ could be mapped to by a permutation that satisfies this compatibility condition. The permutation $\sigma(B_i)$ specifies which one is the actual image point. From this description it is (hopefully) clear that we can associate to every braid $B$ a permutation of $(n^n)^2$ points that is compatible with $\sigma(B)$ on $n^n$ points. It might not be entirely obvious that we have actually constructed an action on the set of $(n^n)^2$ points. 

Since $\theta_n$ and $p$ are a covering maps, the homotopy types of the lifted paths only depend on the homotopy type of the original path in $C_n$. Equivalently, the braid types $B_i$ of the braids formed by the roots of the lifts $f_{i,t}$ are invariants of the original braid $B$. In particular, neither the choice of parametrisation of $B$, nor the choice of parametrisations of the lifted braids $B_i$ as loops in $C_n$ changes the resulting permutation. 

Let $z_0$ through $z_{n^n-1}$ denote the points in the fibre in $Z_n$. Let $\mathbb{B}_n^{n^n}$ be the direct product of $n^n$ copies of $\mathbb{B}_n$. The symmetric group $S_{n^n}$ acts on $\mathbb{B}_n^{n^n}$ by permutation of the components. Algebraically, the lifting procedure gives us a homomorphism $h$ from $\mathbb{B}_n$ to the semidirect product, or wreath product, $\mathbb{B}_n^{n^n}\rtimes S_{n^n}$, where $h$ sends a braid $B$ to the list of lifted braids $B_i$ starting at $z_i$ and the permutation $\sigma(B)$,
\begin{equation}
h:B\mapsto(B_0,B_1,B_2,\ldots,B_{n^n-1},\sigma(B)).
\end{equation} 
It is a homomorphism because addition in $\mathbb{B}_n^{n^n}\rtimes S_{n^n}$ corresponds to concatenation of paths (i.e., braids) with matching start and end points. Composing $h$ with $n^n$ copies of $\sigma$ we obtain a homomorphism $\mathbb{B}_n\to S_{n^n}\rtimes S_{n^n}$. We can now check that the map that sends an element $(\sigma(B_1),\sigma(B_2),\ldots,\sigma(B_{n^n}),\sigma(B))\in S_{n^n}\rtimes S_{n^n}$ to the permutation in Equation $(\ref{eq:perm})$ is a homomorphism $S_{n^n}\rtimes S_{n^n}\to S_{(n^n)^2}$, which altogether results in a homomorphism $\mathbb{B}_n\to S_{(n^n)^2}$. We have
\begin{align}
n^n \sigma(B_{\sigma(A)(i)})(\sigma(A_i)(k)) + \sigma(B)(\sigma(A)(i))&=n^n \sigma(A_i B_{\sigma(A)(i)})(k) + \sigma(AB)(i)\nonumber\\
&=n^n \sigma((AB)_i)(k) + \sigma(AB)(i)
\end{align}
for all $A,B\in\mathbb{B}_n$.
We have therefore constructed an action $\rho_{n,2}$ of $\mathbb{B}_n$ on $(n^n)^2$ points that is compatible with the action $\rho_{n,1}=\sigma$ on $n^n$ points. We iterate this process to actions on $(n^n)^{j+1}$ points that are compatible with the action on $(n^n)^j$ points.



Applying the homomorphism $h$ to each of the $B_i$ results in $n^n$ elements of $\mathbb{B}_n\rtimes S_{n^n}$. Previously, we denoted the image of a braid $B$ under $h$ by $(B_0,B_1,B_2,\ldots,B_{n^n-1},\sigma(B))$. In order to avoid excessive use of subscripts we write $B_{2,i},B_{2,n^n+i},B_{2,(2\times n^n)+i}\ldots,B_{2,(n^n-1)\times n^n+i}$ for the braids that correspond to the lifts of $B_i$ instead of $B_{i_0}$, $B_{i_1}$ and so on. Note in particular that all braids that are lifts of $B_i$ have a second index that is in the residue class $i$ mod $n^n$. We have chosen this labelling such that the map $B\mapsto (B_{2,0},B_{2,1},B_{2,2},\ldots,B_{2,(n^n)^2-1},\rho_{n,2}(B))$ becomes a homomorphism from $\mathbb{B}_n$ to $\mathbb{B}_n^{(n^n)^2}\rtimes S_{(n^n)^2}$. Exactly as in the previous case we define the action of $B$ on $(n^n)^3$ points by
\begin{equation}
(n^n)^2 k+i\mapsto (n^n)^2 \sigma(B_{2,i})(k)+\rho_{n,2}(B)(i),
\end{equation}
where $k=0,1,2,\ldots,n^n-1$ and $j=0,1,2,\ldots,(n^n)^2-1$.

It is important to note that not only the braid types $B_{1,i}=B_i$ of the lifted paths of the parametrisations of $B$ are invariants of $B$. As such the braid types of the lifts of the parametrisations of $B_{2,i}$ are also invariants of $B$, say $B_{2,k n^n+i}$, $k=0,1,2,\ldots,n^n-1$. Continuing like this we obtain a sequence of braids 
\begin{align}
\label{eq:seq}
(B,&\{B_{1,0},B_{1,1},B_{1,2},\ldots,B_{1,n^n-1}\},\{B_{2,0},B_{2,1},B_{2,2},\ldots,B_{2,(n^n)^2-1}\},\ldots,\nonumber\\
&\{B_{j,0},B_{j,1},B_{j,2},\ldots,B_{j,(n^n)^j-1}\},\ldots)
\end{align}
that is an invariant of $B$. This should be understood as follows. If two braids, $A$ and $B$, with sequences
\begin{align}
(A,&\{A_{1,0},A_{1,1},A_{1,2},\ldots,A_{1,n^n-1}\},\{A_{2,0},A_{2,1},A_{2,2},\ldots,A_{2,(n^n)^2-1}\},\ldots,\nonumber\\
&\{A_{j,0},A_{j,1},A_{j,2},\ldots,A_{j,(n^n)^j-1}\},\ldots)
\end{align}
and
\begin{align}
(B,&\{B_{1,0},B_{1,1},B_{1,2},\ldots,B_{1,n^n-1}\},\{B_{2,0},B_{2,1},B_{2,2},\ldots,B_{2,(n^n)^2-1}\},\ldots,\nonumber\\
&\{B_{j,0},B_{j,1},B_{j,2},\ldots,B_{j,(n^n)^j-1}\},\ldots)
\end{align}
are isotopic, then for every ${j,i}$ the braid $A_{j,i}$ is isotopic to $B_{j,i}$.

We can define an action $\rho_{n,j+1}$ on $(n^n)^{j+1}$ points inductively via
\begin{equation}
(n^n)^j k+i\mapsto (n^n)^j \sigma(B_{j,i})(k)+\rho_{n,j}(B)(i),
\end{equation}
where $\rho_{n,j}$ is the action on $(n^n)^j$ points and $k=0,1,2,\ldots,n^n-1$ and $i=0,1,2,\ldots,(n^n)^j-1$. Note that the residue classes mod $(n^n)^j$ are permuted according to $\rho_{n,j}$, i.e., the action on $(n^n)^{j+1}$ points is compatible with the action on $(n^n)^j$ points. We thus obtain a braid group action $\psi_n$ on $\mathbb{Z}_{n^n}\cong\mathbb{Z}_n$. Since $\psi_n$ is constructed from permutations in every coordinate of $\mathbb{Z}_{n^n}$ it preserves the metric and is therefore continuous, both by exactly the same arguments as in the case of $\phi_n$ (cf. Proposition \ref{prop:affine}). This concludes the proof of Theorem \ref{thm:main} \textit{i)}.


\begin{remark}
It should be noted that this new action $\psi_n(\cdot,B)$ carries less (or equal) information about $B$ than $\phi_n(\cdot,B)$. This is because $Z_n$ is a proper subset of $C_n$. In particular, there are non-homotopic paths $\gamma_1$, $\gamma_2$ in $Z_n$ corresponding to families of polynomials $f_t$ and $g_t$ such that the braids that are formed by the roots of $f_t$ and $g_t$ are isotopic. On the other hand, $\psi_n(\cdot,B)$ is completely determined by $h(B)\in \mathbb{B}_n^{n^n}\rtimes S_{n^n}$. 
\end{remark}
\begin{remark}
The construction of both actions in this section is possible because $Z_n$ can be embedded into $V_n$ by sending a polynomial to the set of its non-zero roots. There is another way of embedding $Z_n$ into $V_n$, which sends a polynomial to the set of its critical points. We again obtain a tower of covering spaces $\hat{Z}_n^{j}$ and can define analogues of the actions $\phi_n$ and $\psi_n$. In fact, all results from this section remain true for these actions.

Note that for $j>1$ the spaces $\hat{Z}_n^j$ are different from $Z_n^j$ as they are the spaces of monic polynomials $f$ such that the set of critical values of $f$ is the set of critical points of a polynomial in $\hat{Z}_n^{j-1}$. We will see in Section \ref{sec:comps}, that the analogue of $\psi_3$ is different from $\psi_3$.
\end{remark}

\section{Sequences of braids}
\label{sec:seq}
This section aims at using sequences of braids to make braid invariants stronger. Say we have a way of associating to any braid $B\in\mathbb{B}_n$ a sequence of braids, whose braid types are invariants of $B$. We take a braid invariant $I_n:\mathbb{B}_n\to K_n$ with values in some set $K_n$ that is not very good at distinguishing braids, but easy to compute, such as the exponent sum or the permutation representation, which both grow in complexity linearly with the length of the braid word. We can evaluate $I_n$ on the first $k$ terms of the sequence associated to $B$ and obtain a sequence of braid invariants that is presumably much stronger than the original invariant $I_n$. If two braids $A$ and $B$ are isotopic, then not only are their exponent sums equal, but also every braid in the sequence associated to $A$ is isotopic to its counterpart in the sequence associated to $B$ and hence these exponent sums must be equal too. Furthermore, if it is not hard to compute the sequence corresponding to a braid $B$, then the whole sequence of invariants is relatively easy to compute too.

In the previous section (cf. Eq. (\ref{eq:seq})) we have encountered such a sequence of braids.
There are some other candidates that are just as suitable for our purposes. We could for example consider the collection of lifts in $Z_n^j$ of a loop in $C_n$ that corresponds to the braid $B$. Each of the lifts is a path in $Z_n^j$ and therefore corresponds to a braid $B'_{j,i}$ starting at the point $z_{j,i}$, $i=0,1,2,\ldots,n\times(n^{n-1})^j-1$, in the fiber in $Z_n^j$. We obtain again a sequence of braids 
\begin{equation}
(B,\{B'_{1,i}\}_{i=0,1,2,\ldots,n^n-1},\{B'_{2,i}\}_{i=0,1,2,\ldots,n\times(n^{n-1})^2-1},\ldots)
\end{equation} 
similar to Equation (\ref{eq:seq}), which are invariants of $B$. This time however, the sequence is not determined by the entries ${B'_{1,0},B'_{1,1,},\ldots,B'_{1,n^n-1}}$ and the permutation in $S_{n^n}$. This is essentially because $Z_n^j$ is a proper subset of $V_n$. In particular, there are paths in $Z_n^j$ with the same start and end points that correspond to the same braid (i.e., are homotopic in $C_n$), but are not homotopic in $Z_n^j$.

\begin{remark}
There are two different embeddings of $Z_n^1$ into $V_n$, the one that sends a polynomial to its non-zero roots and the one that sends a polynomial to its critical points. Both of these give rise to an infinite tower of covering spaces $Z_n^j$ and $\hat{Z}_n^j$ and the procedure discussed above works for both of them. Hence we have in fact three sequences of braids, $B_{j,i}$ $B'_{j,i}$ and $\hat{B}_{j,i}$, say, that are candidates for improving braid invariants. 
\end{remark}

\begin{remark}
Note that while the braid sequences are invariants of the braid $B$, they do depend on the base point in $C_n$.
\end{remark}

The sequences $B_{j,i}$, $B'_{j,i}$ and $\hat{B}_{j,i}$ are infinite sequences of braid invariants of $B$. In order to understand if they are useful for the improvement of braid invariants, we need to know how hard it is to compute them. Once we have computed the first $k$ terms of the corresponding sequences and permutations for each generator of $\mathbb{B}_n$, we can calculate the first $k$ terms of the sequences $B'_{j,i}$ and $\hat{B}_{j,i}$ in a number of steps that grows linearly with the length of the braid word of $B$ as it is simply addition in semi-direct products of groups. In the case of $B'_{j,i}$ and $\hat{B}_{j,i}$ the first $k$ terms of the sequences that correspond to a generator can be found by solving $k$ systems of polynomial equations. If $k$ is too large, this becomes impratical. Therefore, even though we have in theory infinite sequences of invariants at our disposal, in practice we will only calculate the first $k$ entries for some relatively low number $k$. In the case of $B_{j,i}$ we only have to solve one system of polynomial equations to find the first term of the sequence $\{B_{1,i}\}_{i=0,1,2\ldots,n^n-1}$ corresponding to a generator $B=\sigma_m$. From this first term we can calculate $B_{j,i}$ for any braid $B$ and any $j$, $i$ in a number of steps that grows linearly with the length of $B$. The sequence $B_{j,i}$ is therefore significantly easier to compute than $B'_{j,i}$ or $\hat{B}_{j,i}$.

The braids $B_{j,i}$, $B'_{j,i}$ and $\hat{B}_{j,i}$ in the sequences are open braids in the sense that in general the corresponding paths in $Z_n$ and $Z_{n}^j$ respectively are not necessarily loops. We can associate braids to the paths nonetheless, since isotopies of the braids formed by the roots of the polynomials that make up the paths are equivalent to homotopies of paths that keep the start and end points fixed. However, there is also a way to associate to a braid $B$ a sequence of braids $B_{j,C}$ that correspond to loops in $Z_n^j$. These can be used to turn an invariant of conjugacy classes of braids into a sequence of invariants of conjugacy classes of braids, which is presumably much stronger.

Denote by $\phi_{n,j}(B)$ the restriction of the action $\phi_n(\cdot,B)$ to an action on $n\times(n^{n-1})^j$ points and let $\mathfrak{C}_{j}$ the set of cycles of $\phi_{n,j}(B)$. Then every cycle
\begin{equation}
C=(i\ \ \ \phi_{n,j}(B)(i)\ \ \ \phi_{n,j}(B)^2(i)\ \ \ldots\ \ \phi_{n,j}(B)^{|C|}(i))\in\mathfrak{C}_j,
\end{equation} 
where $|C|$ is the length of the cycle $C$, can be associated with the braid corresponding to the loop $\gamma_{C}$ in $Z_n^j$ that is the concatenation of the lifts starting at $z_{j,i}$, $z_{j,\phi_{n,j}(B)(i)}$, $z_{j,\phi_{n,j}(B)^2(i)}$ and so on, where we denote the points in the fibres in $Z_n^j$ by $z_{j,i}$. 

We thus obtain a sequence of braids 
\begin{equation}
\label{eq:conjseq}
(B,\{B_{C_{1,i}}\}_{i=1,2,\ldots,|\mathfrak{C}_1|},\{B_{C_{2,i}}\}_{i=1,2,\ldots,|\mathfrak{C}_2|},\ldots)
\end{equation}
where we have labelled the cycles in $\mathfrak{C}_j$ by $C_{j,i}$ and $|\mathfrak{C}_j|$ denotes the number of cycles of $\phi_{n,j}(B)$. 

The above procedure again really gives rise to several sequences of braids, depending on how we choose to associate a braid to a path in $Z_n^j$, the braid that is formed by its roots $B'_{j,i}$ or by its critical points $\hat{B}_{j,i}$.



Now every braid in the sequence (\ref{eq:conjseq}) (apart from $B$) comes from a loop in $Z_n^j$, which was obtained from a lifting procedure. This means that the sequence (\ref{eq:conjseq}) is an invariant of the conjugacy class of $B$ in the sense that if $A$ is conjugate to $B$, then for every $j$ there is a bijection $g_j$ between the sets of cycles of $\phi_{n,j}(B)$ and $\phi_{n,j}(A)$ that maps each cycle to a cycle of the same length. Furthermore, $A_{g_j(C_{j,i})}$ is conjugate to $B_{C_{j,i}}$. 
We can now apply an invariant $J_n$ of braid conjugacy classes in $\mathbb{B}_n$ to the sequence in Eq. (\ref{eq:conjseq}) and obtain a sequence of invariants that is (presumably) a lot stronger than the original invariant $J_n$. 

At this moment it is not clear how such a sequence of braids changes under stabilisation and destabilisation, i.e., how the sequence corresponding to $B\sigma_n^{\pm1}$ is related to the sequence corresponding to a $n$-strand braid $B$. The hope is that there are some properties that stay invariant, which can be used to define a sequence of link invariants analogously to the sequence of invariants of braids and braid conjugacy classes in the previous paragraphs.

\section{Computations for $n=2$ and $n=3$}
\label{sec:comps}

In this section we compute how the generators of the braid groups $\mathbb{B}_n$ act on $\mathbb{Z}_{n^n}\cong \mathbb{Z}_n$, for $n=2$ and $n=3$ via $\psi_n$. The effect of more complicated braid words on two or three strands can then be obtained by composing the permutations coming from the relevant generators.

\subsection{The case of $n=2$}
The case of $n=2$ is probably the only case that is simple enough to be done by hand. One possible parametrisation of the only generator $\sigma_1$ of $\mathbb{B}_2\cong\mathbb{Z}$ is given by $(\tfrac{1}{2}\rme^{\rmi(t+\epsilon)},-\tfrac{1}{2}\rme^{\rmi(t+\epsilon)})$, where $t$ is going from $0$ to $\pi$ and $\epsilon$ is some small positive real number. In theory choosing $\epsilon=0$ is possible, but it will become clear that this is not a good choice. This parametrisation is a loop in $C_2$ with base point $(\tfrac{1}{2}\rme^{\rmi\epsilon},-\tfrac{1}{2}\rme^{\rmi\epsilon})$ and its lifts in $V_2$ are given by $x(t)=\rme^{\rmi (t+\epsilon)}$ and $y(t)=-\rme^{\rmi (t+\epsilon)}$, two paths permuting the 2 points $x_1=\rme^{\rmi\epsilon}$ and $x_2=-\rme^{\rmi\epsilon}$ in $V_2$.

In order to compute the action we first need to find the preimage points of $x_1$ and $x_2$ under $\theta_2$, i.e., we have to find monic quadratic polynomials with one (simple) root equal to 0 and the critical value equal to $x_1=\rme^{\rmi \epsilon}$ (and $x_2=-\rme^{\rmi \epsilon}$, respectively). We thus have to solve
\begin{align}
c(c-z)&=x=\rme^{\rmi \epsilon}\nonumber\\
2c-z&=0
\end{align}
as well as
\begin{align}
c(c-z)&=y=-\rme^{\rmi \epsilon}\nonumber\\
2c-z&=0
\end{align}
for the non-zero root $z$. In the first case, we obtain $z=\pm 2\rme^{\rmi (\epsilon+\pi)/2}$, which are the $n^{n-1}=2$ preimage points of $x_1$, corresponding to the polynomials $u(u\mp 2\rme^{\rmi (\epsilon+\pi)/2})\in Z_2$. In the latter case, we have $z=\pm 2\rme^{\rmi \epsilon/2}$. We label these points as follows: $z_1=2\rme^{\rmi \epsilon/2}$, $z_2=2\rme^{\rmi(\epsilon+\pi)/2}$, $z_3=-2\rme^{\rmi \epsilon/2}$ and $z_0=-2\rme^{\rmi (\epsilon+\pi)/2}$. The reason for our choice of a positive $\epsilon$ was that the real parts of the points in this fibre are non-zero, which allows us to read off the braid words corresponding to lifted paths. 
Now we calculate how the two paths $x(t)$ and $y(t)$ in $V_2$ permute these four preimage points. We already know that the preimage points of $x_1$ get mapped to preimage points of $x_2$ and vice versa.

We solve
\begin{align}
c(t)(c(t)-z(t))&=x(t)=\rme^{\rmi (t+\epsilon)}\nonumber\\
2c(t)-z(t)&=0
\end{align}
as well as
\begin{align}
c(t)(c(t)-z(t))&=y(t)=\rme^{\rmi (t+\epsilon+\pi)}\nonumber\\
2c(t)-z(t)&=0
\end{align}
for the non-zero root $z(t)$. In the first case, we obtain $z(t)=2\rme^{\rmi (t+\epsilon\pm\pi)/2}$ and in the latter $z(t)=\pm2\rme^{\rmi (t+\epsilon)/2}$. The lifted path that starts at $z_1=2\rme^{\rmi \epsilon}$ is $2\rme^{\rmi (t+\epsilon)/2}$ and it ends at $t=\pi$ at $2\rme^{\rmi (\epsilon+\pi)/2}$, which is $z_{2}$. The lifted path that starts at $z_{2}=2\rme^{\rmi (\epsilon+\pi)/2}$ is $2\rme^{\rmi (t+\epsilon+\pi)/2}$ and it ends at $t=\pi$ at $2\rme^{\rmi (\epsilon/2+\pi)}=-2\rme^{\rmi \epsilon/2}=z_3$. Since we know that all lifted paths that start at preimage points of $x_1$ end at preimage points of $x_2$ and vice versa, this is enough to conclude that the permutation sends $z_{i}$ to $z_{i+1}$, where the index is taken mod $4$. Hence the cycle notation of the action of $\sigma_1$on four points is
\begin{equation}
\rho_{2,1}(\sigma_1)=(0\ \ 1\ \ 2\ \ 3).
\end{equation}

We could interpret each of the four paths as paths in $V_2$ and lift them too, which would tell us the action of $\sigma_1$ on the third coordinate of $\mathbb{Z}_{2}$ via $\phi_2$, i.e., the restriction of $\phi_2$ to an action on eight points. However, this would require a calculation of the eight preimage points of $\{z_0,z_1,z_2,z_3\}$, so after a number of iterations this procedure becomes unnecessarily long. Instead we are going to focus on the braids that each lifted path corresponds to, as in the construction of $\psi_2$. Recall that each of the lifted paths corresponds to a family of polynomials, one of whose roots is 0 and the other given by $z(t)$. The points $x_1$ and $x_2$ were chosen such that their preimage points are not purely imaginary, which makes it possible to read off the braid word from a parametric plot of the roots corresponding to a lifted path, namely the 2-strand braid that is formed out of the 0-strand and $z(t)$.

For the paths that start at $z_1$ and $z_3$ we find that the corresponding braid is $\sigma_1$ and for the paths that start at $z_2$ and $z_0$ we obtain the trivial braid (cf. Figure \ref{fig1}). In the notation of the previous section, $B_{1,1}=B_{1,3}=\sigma_1$ and $B_{1,2}=B_{1,0}=e$. Each of these braids can be parametrised as a loop in $C_2$, which leads to paths in $V_2$ with start and end points at $\{x_1,x_2\}$. This information is enough to compute the action of any 2-strand braid $B$ on $\mathbb{Z}_2$.
We would like to give a bit more insight into the calculations and also compute the action of $\sigma_1$ on the second coordinate of $\mathbb{Z}_4$, i.e., on 16 points.
For each of the four braids $B_{1,j}$ we obtain again an action on the preimage points of $\{x_1,x_2\}$, a trivial permutation for the trivial braids and cyclic permutations $(0\ 1\ 2\ 3)$ for the $\sigma_1$ as above.

\begin{figure}
\centering
\labellist
\large
\pinlabel a) at -50 400
\pinlabel b) at 600 400
\pinlabel c) at -50 -10
\pinlabel d) at 600 -10
\endlabellist
\includegraphics[height=5cm]{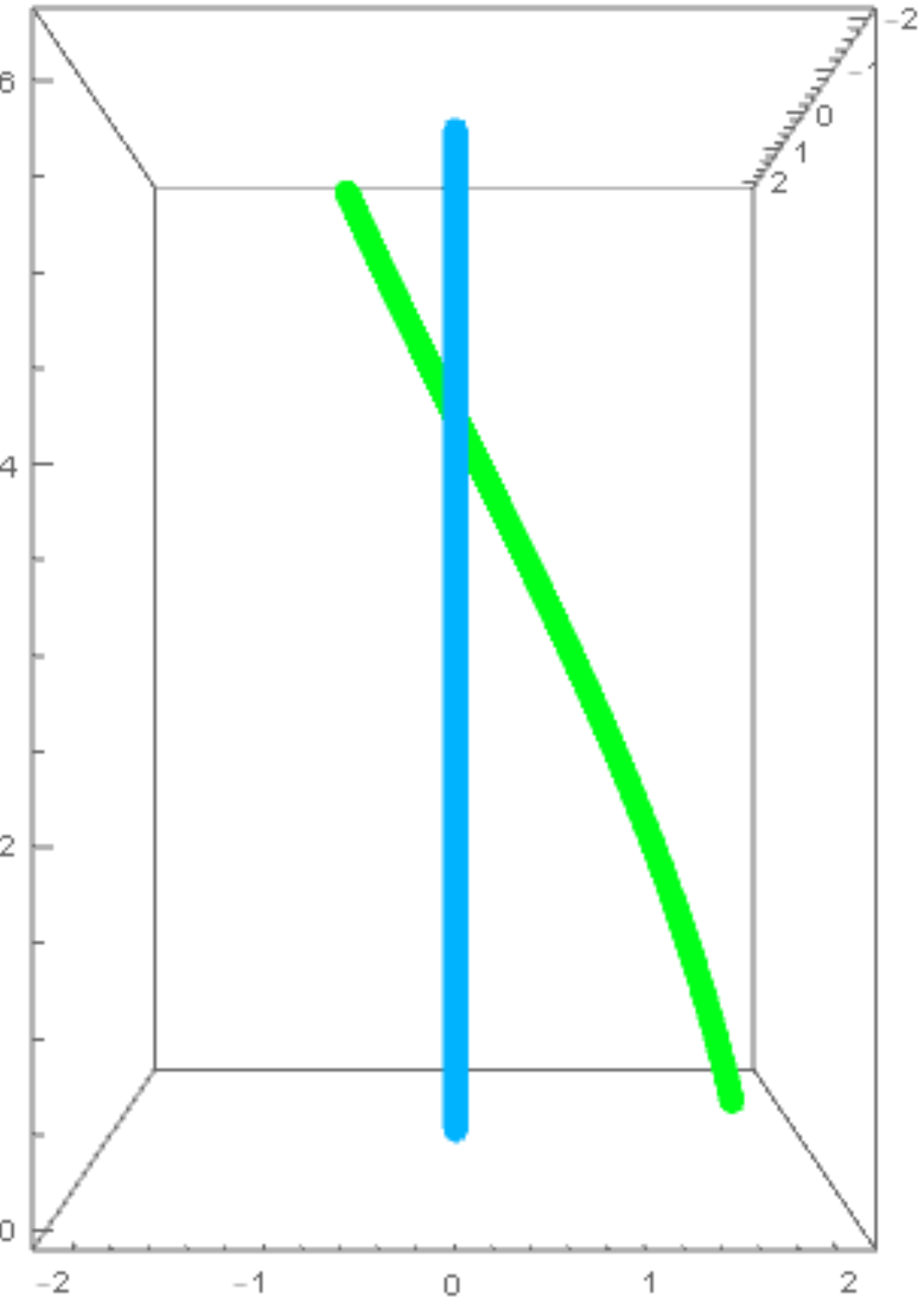}\quad\quad
\labellist
\large
\pinlabel b) at -20 400
\pinlabel d) at -20 -10
\endlabellist
\includegraphics[height=5cm]{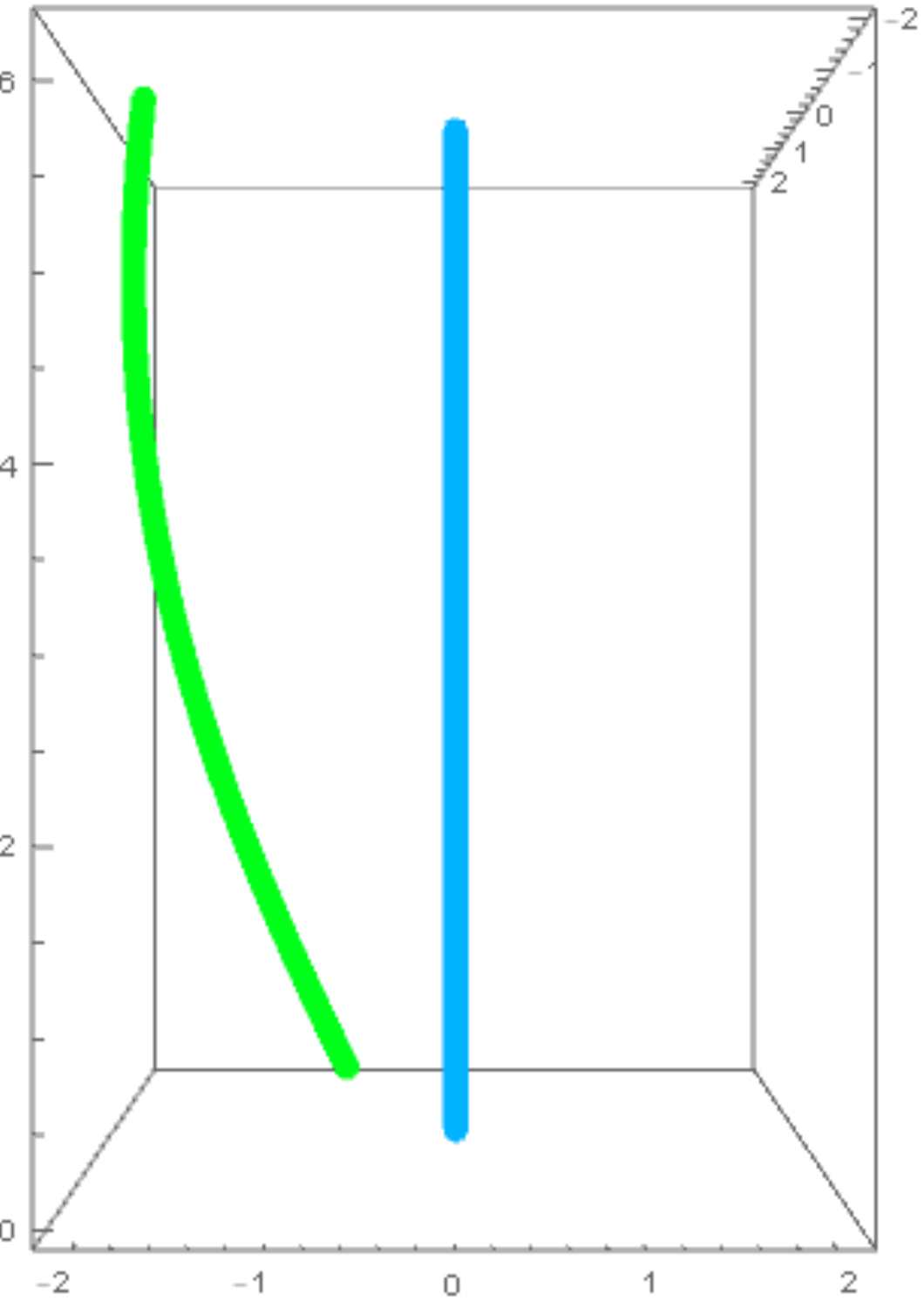}\\
\includegraphics[height=5cm]{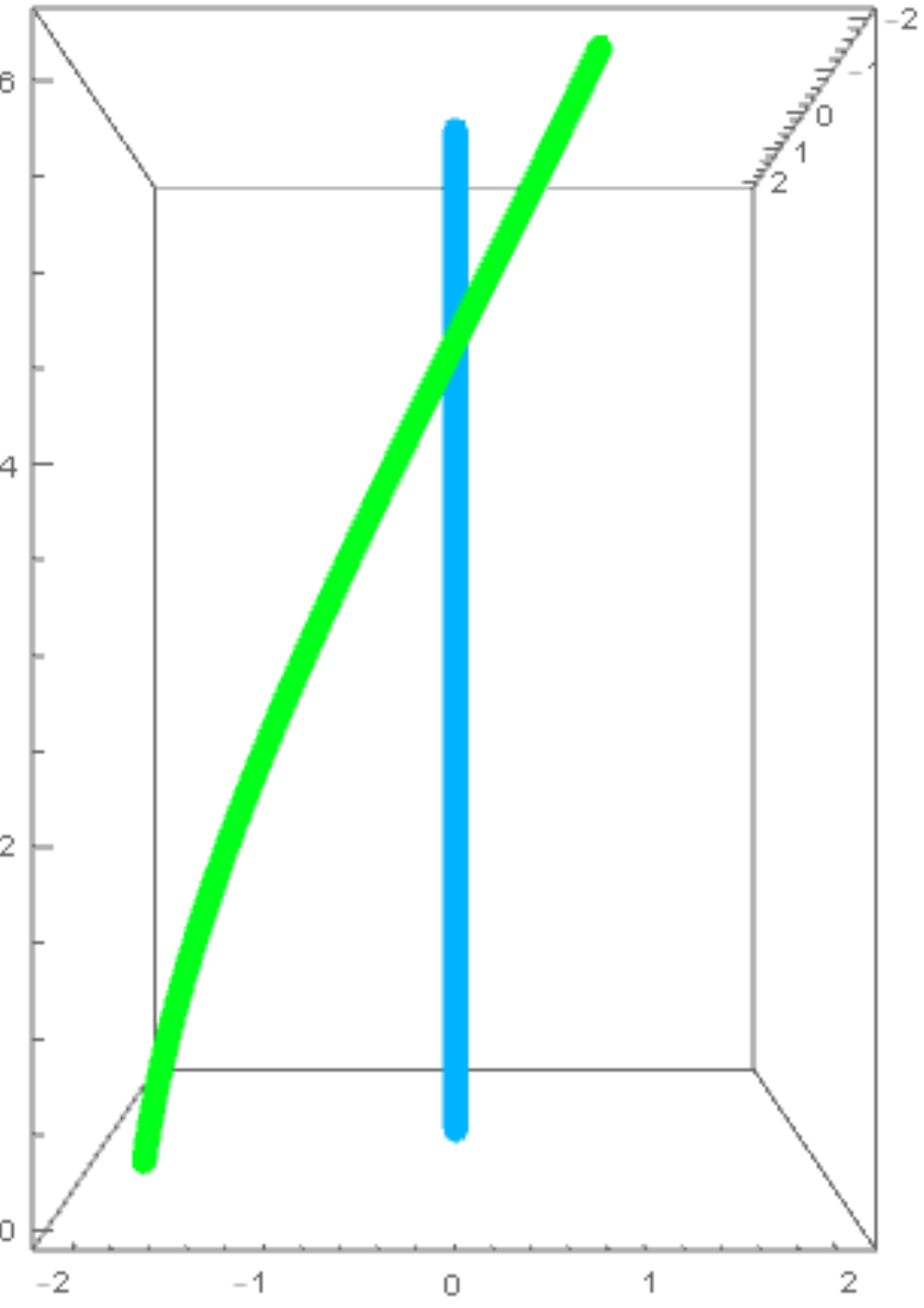}\quad\quad
\includegraphics[height=5cm]{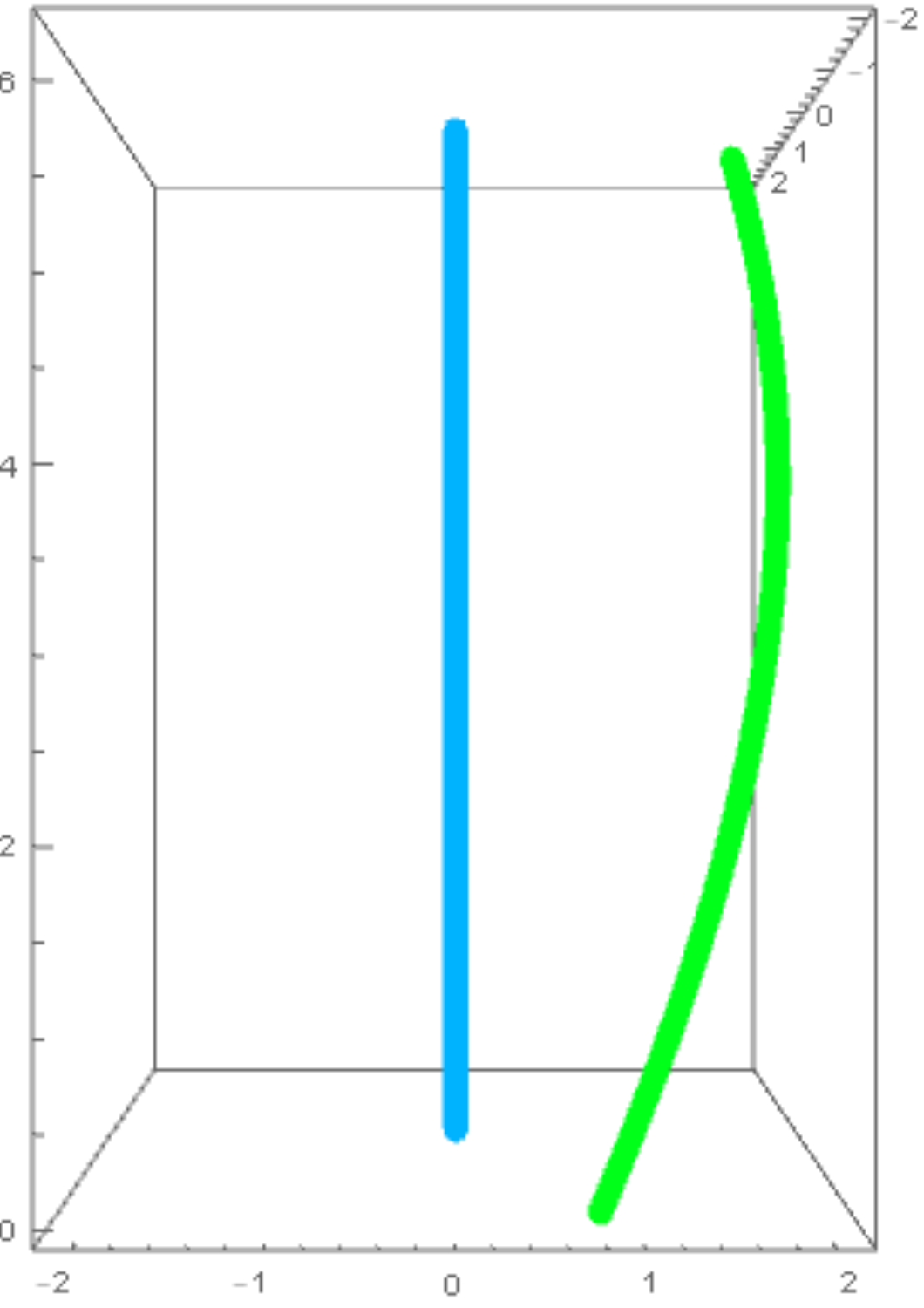}
\caption{\label{fig1}The braids that are formed by the roots of the polynomials corresponding to the lifts of $\sigma_1$ in $V_2$ with $\epsilon=0.8$. a) The lift that starts at $z_1$ has the corresponding braid $B_{1,1}=\sigma_1$. b) The lift that starts at $z_2$ corresponds to the trivial braid, $B_{1,2}=e$. c) The lift that starts at $z_3$ has the corresponding braid $B_{1,3}=\sigma_1$. d) The lift that starts at $z_0$ corresponds to the trivial braid, $B_{1,0}=e$. }
\end{figure}

We arrange these four permutations of four points into one permutation of 16 points, labelled 0 through 15, as indicated in Section \ref{sec:actions}. Points that have a label $j$ mod 4 must be mapped to a point with a label $j+1$ mod 4 in order to be consistent with the action that $\sigma_1$ induces on 4 points. Therefore, the point with the label 1 can go to four possible values, 2, 6, 10 and 14. The label 1 corresponds to the braid $B_{1,1}=\sigma_1$, which induces the cyclic permutation on four points. We thus send the point 1, the first of the points 1 mod 4, to the second possible point that it can go to, in this case 6. Similarly, the point 5, the second of points with a label 1 mod 4, is mapped to 10, which is the third possible value that it can take. In exactly the same way 9 is mapped to 14, 13 to 2, 3 (the first point 3 mod 4, which is the other residue class $j$ mod 4 whose corresponding braid is $B_{1,j}=\sigma_1$) to 4 (the second point 0 mod 4), 7 to 8, 11 to 12 and 15 to 16.
The points that have even labels are start points of the lifts $B_{1,2}$ and $B_{1,0}$, both of which are trivial and therefore induce trivial permutations. Therefore, the point 2, which is the first in the list of points with label 2 mod 4 should be mapped to the first point of the list of points 3 mod 4, i.e., 3. and so on.

In other words, the permutation on $4^j$ points dictates how the residue classes mod $4^j$ are permuted. Therefore, every point $x\in\{0,1,2,\ldots,4^{j+1}-1\}$ has 4 possible image values. Which one of these it is mapped to is determined by the permutation on 4 points that is induced by the braid $B_{j,x\text{ mod }4^j}$.
 
Therefore the permutation of 16 points induced by the generator $\sigma_1$ is given by
\begin{align}
4k+1&\mapsto 4(k+1) +2 &\text{ mod }16,\nonumber\\
4k+2&\mapsto 4k+3 &\text{ mod }16,\nonumber\\
4k+3&\mapsto 4(k+1) &\text{ mod }16,\nonumber\\
4k&\mapsto 4k+1 &\text{ mod }16,
\end{align}
where $k=0,1,2,3$. In cycle notation this is
\begin{equation}
\label{eq:16}
(0\ 1\ 6\ 7\ 8\ 9\ 14\ 15)(2\ 3\ 4\ 5\ 10\ 11\ 12\ 13).
\end{equation}

For a general 2-strand braid $B=\sigma_1^k$, $k\in\mathbb{Z}$ the action on $\mathbb{Z}_2$ can now be described as follows. For the first coordinate of $\mathbb{Z}_2$ there are two points, 0 and 1, and $B$ permutes non-trivially if and only if $k$ is odd. The action of $B$ on the four possible choices of the second coordinate is given by $\tau^k$, where $\tau$ is the cyclic permutation of four elements, $(1\mapsto2\mapsto3\mapsto0\mapsto1)$. Note that for a given parity of $k$ there are again only two possible permutations and the action is trivial if and only if $k$ is divisible by four. The braid words that are formed by the four lifted paths are each $\sigma_1^{k/2}$ if $k$ is even and twice $\sigma_1^{\lfloor k/2\rfloor}$ and twice $\sigma_1^{\lceil k/2\rceil}$ if $k$ is odd. Using the notation of Section \ref{sec:actions}, we have 
\begin{equation}
B_{1,1}=B_{1,2}=B_{1,3}=B_{1,0}=\sigma_1^{k/2}
\end{equation} 
if $k$ is even and 
\begin{align}
B_{1,1}=B_{1,3}&=\sigma_1^{\lceil k/2\rceil},\nonumber\\
B_{1,2}=B_{1,0}&=\sigma_1^{\lfloor k/2\rfloor},
\end{align}
if $k$ is odd.


If we take for example $B=\sigma_1^{12}$, then the action on the four preimage points of $\{x_1,x_2\}$ is given by the twelfth power of the cyclic permutation, i.e., the identity. The four lifted paths each has the braid $\sigma_1^6$. Each of these paths gives us a permutation of the four preimage points of $\{x_1,x_2\}$. Since all four braids are identical, we obtain the same permutation every time, the 6th power of the cyclic permutation, i.e., $(1\mapsto3\mapsto1, 2\mapsto4\mapsto2)$. These can be arranged into a permutation of 16 elements that is compatible with $\tau_1(B)=\text{id}$, namely
\begin{equation}
\rho_{2,2}(\sigma_1^{12})=(1\  9)(5\  13)(2\ 10)(6\ 14)(3\ 11)(7\ 15)(4\ 12)(8\ 0).
\end{equation} 
This permutation was obtained as follows. Because $\tau_1(B)=\text{id}$ every number must be mapped to a number which lies in the same residue class mod 4. Since we are constructing a permutation on 16 elements, there are four possible images of each number, for example 1 could go to 1, 5, 9 or 13. The permutation $(1\mapsto3\mapsto1, 2\mapsto0\mapsto2)$ now tells us how these four numbers are permuted, namely the first and the third number in this residue class (1 and 9) are swapped and the second and the fourth number in this residue class (5 and 13) are swapped. Analogous computations apply to the other residue classes. 

The braid words of the 16 lifted paths are each $\sigma_1^3$, which we can use in a similar fashion to construct the action on 64 elements as
\begin{equation}
16k+i\mapsto 16(k+3)+\rho_{2,2}(\sigma_1^{12})(i),
\end{equation}
since $\sigma(\sigma_1^3)=(0\ 1\ 2\ 3)^3$. The braids corresponding to the lifts of $\sigma_1^3$ are either $\sigma_1^2$ or $\sigma_1$ depending on the start point. We can use these to compute the action on 256 points and then $4^j$ points by repeating the process.


 
\begin{proposition}
\label{prop:free}
The action $\psi_2$ of $\mathbb{B}_2$ on $\mathbb{Z}_2$ is faithful.
\end{proposition} 
\begin{proof}
Consider a braid $B=\sigma_1^k$ with $k\in\mathbb{Z}$ and $k\neq 0$, i.e., $k=2^m q$ for some $m\in\mathbb{N}_0$ and some odd $q\in\mathbb{Z}$. As we have seen above, the induced permutation is non-trivial if $m=0$. If $m>0$, the computations from the previous paragraphs tell us that after the $m-1^{\text{th}}$ iteration of the lifting procedure results in $4^{m-1}$ paths in $V_2$ that each form the braid $\sigma_1^{2q}$. In the notation of the previous section, we have $B_{m,j}=\sigma_1^{2q}$ for all $j=1,2,\ldots,4^m$.

This means that the permutation on the $m^{\text{th}}$ coordinate of $\mathbb{Z}_{n^n}$ is non-trivial. For example, the elements 1, $4^{m-1}+1$, $2\times4^{m-1}+1$ and $3\times 4^{m-1}+1$ are permuted in a non-trivial way, namely 
\begin{equation}
(1\mapsto4^{m-1}+1\mapsto2\times4^{m-1}+1\mapsto3\times4^{m-1}+1)^{2q\text{ mod }4},
\end{equation}
where $q$ is odd.
Therefore, the trivial braid (with $k=0$) is the only braid that leads to a trivial permutation of the $n$-adic integers. In other words, the action on $\mathbb{Z}_4$ is faithful and hence $\psi_2$ is faithful on $\mathbb{Z}_2$.
\end{proof}

We can see from the calculations that the action is not transitive. Recall from Equation (\ref{eq:16}) the permutation of 16 points that is induced by the generator $\sigma_1$, i.e., the action of $\sigma_1$ on the second coordinate of $\mathbb{Z}_{4}$. Since there are two distinct cycles, no power of $\sigma_1$ can map a $4$-adic integer whose second coordinate is 1 to a 4-adic integer whose second coordinate is 2. 

\begin{remark}
Note that the fact that the critical points of a polynomial in $Z_n$ are branch points with deficiency 2 is reflected in the halving of the exponent of $\sigma_1$ with each lifting procedure.
\end{remark}

\begin{remark}
We mentioned before that we can use the other embedding of $Z_n$ into $V_n$, the one that sends a polynomial to its set of critical points rather than its non-zero roots, to define actions on $\mathbb{Z}_n$ completely analogously to the constructions in Section \ref{sec:actions}. We can perform similar calculations for the  analogue of $\psi_n$ as above. In principle, we could obtain different permutations and lifted braids (see the case of $n=3$). However, in the case of $n=2$, the critical point of a monic polynomial with constant term equal to 0 is precisely half of the non-zero root. Hence, we obtain the same permutations and braids.
\end{remark}

\subsection{The case of $n=3$}

For the case of $n=3$ we have to consider the two generators $\sigma_1$ and $\sigma_2$. As a base point in $C_3$ we choose (quite arbitrarily) $\{0,\rme^{\rmi \pi /4}, 2\}$. The correspondingly chosen points in $V_3$ are therefore $x_1=\{\rme^{\rmi \pi/4}, 2\}$, $x_2=\{-\rme^{\rmi \pi/4},2-\rme^{\rmi \pi/4}\}$ and $x_3=\{-2,\rme^{\rmi\pi/4}-2\}$. Note that the labels are chosen such that $x_i$ has 0 in $i^{\text{th}}$ position.



The preimage set $\theta_3^{-1}(\{x_1,x_2,x_3\})$ consists of $3^{3-1}=27$ points in $Z_3$. These are, tuples of complex numbers, which can be identified by solving the corresponding system of polynomial equations, $\{f(c_1),f(c_2)\}=x_j$, $f'(c_i)=0$ for the roots of $f$, as
\begin{align}
z_0&=\{-2.43868-\rmi0.93710, -1.95413+\rmi0.64884\},\nonumber\\
z_1&=\{2.01732+\rmi0.49143, 2.56659-\rmi0.71767\},\nonumber\\
z_2&=\{1.59961-\rmi0.51396, 2.81046+\rmi0.04969\},\nonumber\\
z_3&=\{0.48455+\rmi1.58593, 2.43868+\rmi0.93710\},\nonumber\\
z_4&=\{1.32174-\rmi0.12887, 1.90481+\rmi1.86390\},\nonumber\\
z_5&=\{1.09356-\rmi0.76680, 1.44826-\rmi2.40909\},\nonumber\\
z_6&=\{0.41516-\rmi2.01675, 2.03089-\rmi1.64342\},\nonumber\\
z_7&=\{0.77247+\rmi1.08022, 1.43425-\rmi1.50134\},\nonumber\\
z_8&=\{0.11729+\rmi1.33045, 1.36220+\rmi2.45878\},\nonumber\\
z_9&=\{0.40779+\rmi2.58051, 1.53898+\rmi1.36791\},\nonumber\\
z_{10}&=\{-0.77247-\rmi1.08022, 0.66178-\rmi2.58157\},\nonumber\\
z_{11}&=\{-0.11729-\rmi1.33045, 1.24491+\rmi1.12832\},\nonumber\\
z_{12}&=\{-0.48455-\rmi1.58593, 1.95413-\rmi0.64884\},\nonumber\\
z_{13}&=\{-1.32174+\rmi0.12887, 0.58308+\rmi1.99277\},\nonumber\\
z_{14}&=\{-1.59961+\rmi0.51396, 1.21085+\rmi0.56365\},\nonumber\\
z_{15}&=\{-0.41516+\rmi2.01675, 1.61573+\rmi0.37333\},\nonumber\\
z_{16}&=\{-2.01732-\rmi0.49143, 0.54927-\rmi1.20909\},\nonumber\\
z_{17}&=\{-1.09356+\rmi0.76680, 0.35470-\rmi1.64229\},\nonumber\\
z_{18}&=\{-0.40780-\rmi2.58051, 1.13118-\rmi1.21260\},\nonumber\\
z_{19}&=\{-1.43425+\rmi1.50134, -0.66178+\rmi2.58157\},\nonumber\\
z_{20}&=\{-1.44826+\rmi2.40909, -0.35470+\rmi1.64229\},\nonumber\\
z_{21}&=\{-1.53898-\rmi1.36791, -1.13118+\rmi1.21260\},\nonumber\\
z_{22}&=\{-1.90481-\rmi1.86390, -0.58308-\rmi1.99277\},\nonumber\\
z_{23}&=\{-1.36220-\rmi2.45878, -1.24491-\rmi1.12832\},\nonumber\\
z_{24}&=\{-2.03089+\rmi1.64342, -1.61573-\rmi0.37333\},\nonumber\\
z_{25}&=\{-2.56659+\rmi0.71767, -0.54927+\rmi1.20909\},\nonumber\\
z_{26}&=\{-2.81046-\rmi0.04969, -1.21085-\rmi0.56365\}.
\end{align}

Numbers are rounded to five decimal points. The labels are such that the preimage points of $x_i$ have a label from the residue class $i\text{ mod }3$. 

The preimage set of one fixed $x_i$ has some symmetries. If $y=(y_1,y_2)\in Z_3$ maps to $x_i$, so does $\rme^{2\pi\rmi/3}y=(\rme^{2\pi\rmi/3}y_1,\rme^{2\pi\rmi/3}y_2)$. This, or rather its analogue with $\rme^{2\pi\rmi/n}$, is in fact easy to check for polynomials of any degree $n$, not just $n=3$. The other symmetry is a bit more mysterious (at least to the author). It seems like if $y=(y_1,y_2)\in Z_3$ maps to $x_i$, there is another point $y'=(y_1',y_2')\in Z_3$, which maps to $x_i$ and has $y_1'$ equal to $-y_1$. Furthermore, the point $(-y_2,-y_2')$ is also in $Z_3$ and also maps to $x_i$. The few examples that we have explicitly studied seem to suggest that with these symmetries all preimage points of $x_i$ can be calculated starting from one preimage point $y\in\theta_3^{-1}(x_i)$. 

For a given parametrisation of $\sigma_1$ as a loop in $C_3$ for example
\begin{equation}
\{y_1(t),y_2(t),y_3(t)\}=\{\frac{1}{2}\rme^{\rmi (\pi/4+t)}+\frac{1}{2}\rme^{\rmi \pi/4},-\frac{1}{2}\rme^{\rmi (\pi /4+t)}+\frac{1}{2}\rme^{\rmi \pi/4},2\},
\end{equation}
where $t$ is going from 0 to $\pi$, the corresponding paths $\{v_1(t),v_2(t)\}$ in $V_3$ with start and endpoints in $\{x_1,x_2,x_3\}$ are $\{y_2(t)-y_1(t),y_3(t)-y_1(t)\}$, $\{y_1(t)-y_2(t),y_3(t)-y_2(t)\}$ and $\{y_1(t)-y_3(t),y_2(t)-y_3(t)\}$. 

The paths permute the start and endpoints $x_i$ according to the permutation representation of $\sigma_1$, i.e., the path that starts at $x_1$ ends at $x_2$ and vice versa, while the path that starts at $x_3$ is a loop. We can lift these paths in $V_3$ to paths in $Z_3$ with start and end points in $\theta_3^{-1}(\{x_1,x_2,x_3\})$. We do this by solving a system of polynomial equations,
\begin{align}
f_t(c_k)=c_k(t)\prod_{i=1}^2(c_k(t)-u_i(t))&=v_k(t),&k=1,2,\nonumber\\
\frac{\partial f_t}{\partial u}(c_k(t))&=0,&k=1,2,\nonumber
\end{align}
for $u_i$ for values $t=j\pi/100$, $j=0,1,\ldots,100$ for each of the three paths. With a set of 101 data points for each strand, it is easy to identify the lifted paths in $Z_3$ and the corresponding braids that are formed by the roots of the polynomials that make up the paths.

\begin{figure}[h]
\centering
\labellist
\large
\pinlabel a) at -100 400
\pinlabel b) at 350 400
\endlabellist
\includegraphics[height=5cm]{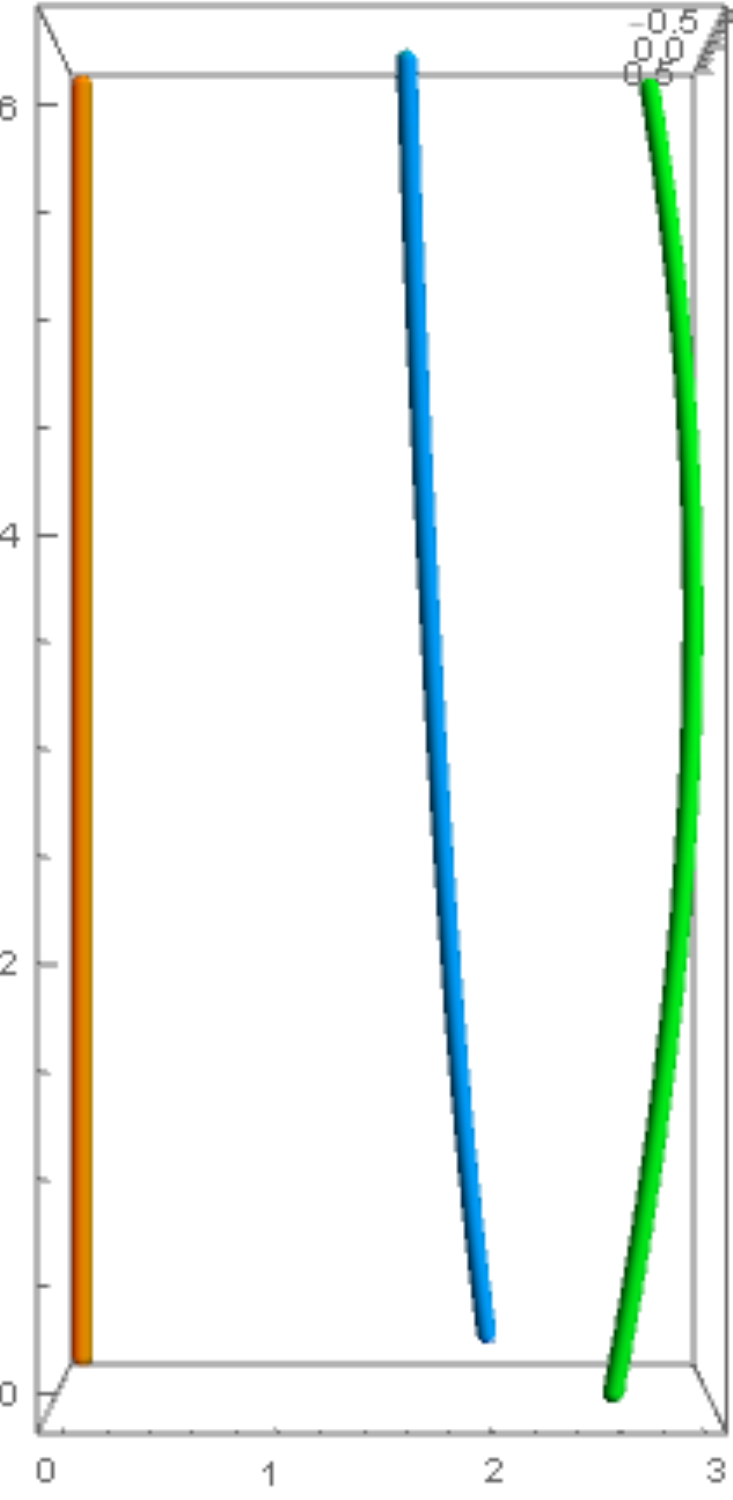}\qquad\qquad\qquad
\includegraphics[height=5cm]{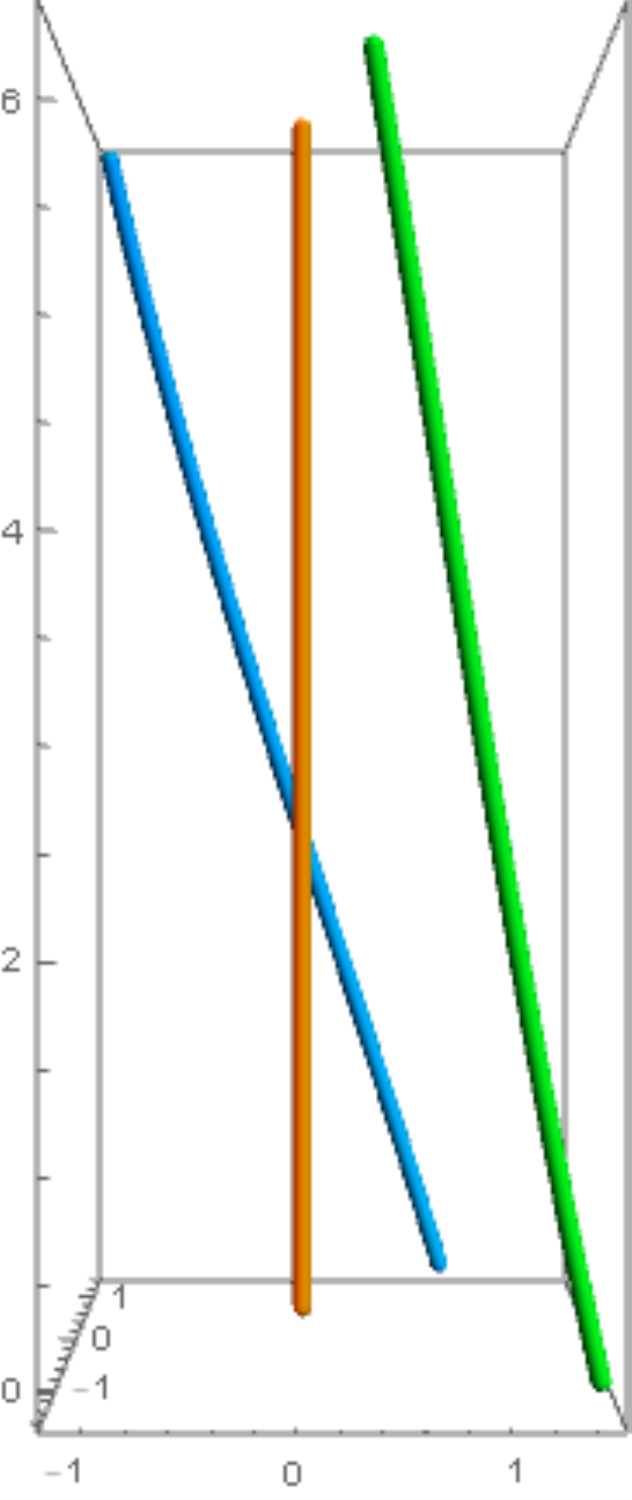}
\caption{\label{fig2}Braids that are formed by the roots of the polynomials corresponding to the lifts of $\sigma_1$ in $Z_3$. a) $B_{1,1}$ is the trivial braid. The configuration of the roots (strands) at $t=0$ is that of $z_1$ together with the orange strand as the $0$-strand. At $t=2\pi$ the non-zero strands are in the position of $z_2$. b) $B_{1,3}=\sigma_1$. It starts at $z_3$ and ends at $z_{9}$.}
\end{figure}

For example, the lift that starts at $z_1$ ends at $z_2$. A braid that interpolates the 101 data points for the set of roots of the polynomials that form this lift is shown in Figure \ref{fig2}a) and it is obviously the trivial braid. Performing this lifting procedure for every $z_i$ gives us the permutation of the 27 points in $\theta_3^{-1}(\{x_1,x_2,x_3\})$, namely

\begin{align}
\rho_{3,1}(\sigma_1)=&(1\ 2)(3\ 9\ 15)(4\ 8\ 13\ 11)(5\ 7\ 17\ 10)(6\ 12\ 18)(14\ 25\ 26\ 16)\nonumber\\
&(19\ 20)(21\ 24\ 0)(22\ 23).
\end{align}

Note that because $\pi(\sigma_1)=(1\ 2)$, every lift that starts at a point with index 1 mod 3 must end at a point with index 2 mod 3 and vice versa. Therefore the permutation above can be split into those cycles that only contain numbers 0 mod 3 and those that permute numbers that are not 0 mod 3.

From the lifted paths we find the braid words of the corresponding braids
\begin{align}
\label{eq:aas}
(A_{1,0},A_{1,1},A_{1,2},\ldots,A_{1,26})=(&e,e, \sigma_2, e, e, e, \sigma_1^{-1}, \sigma_1,\sigma_1,\sigma_1,\sigma_1,\sigma_1,e,e,\sigma_2,\sigma_1^{-1},\nonumber\\
&e,e,\sigma_1,\sigma_1,e,\sigma_1,e,\sigma_1,\sigma_1^{-1},e,\sigma_2)
\end{align}
where $e$ denotes the trivial braid and $A_{1,i}$ is the braid corresponding to the lift that starts at $z_i$.


Similarly for $\sigma_2$ we obtain the permutation
\begin{align}
\rho_{3,1}(\sigma_2)=&(1\ 4\ 7)(2\ 3\ 14\ 12)(5\ 6)(8\ 9)(10\ 16\ 22)(11 \ 21\ 23\ 18)\nonumber\\
&(13\ 19\ 25)(15\ 20\ 24\ 17)(26\ 0)
\end{align}

Note that because $\pi(\sigma_2)=(2\ 3)$, every lift that starts at a point with index 2 mod 3 must end at a point with index 0 mod 3 and vice versa. Therefore the permutation above can be split into those cycles that only contain numbers 1 mod 3 and those that permute numbers that are not 1 mod 3.

The braids corresponding to the lifts are
\begin{align}
\label{eq:bbs}
(B_{1,0},B_{1,1},B_{1,2},\ldots,B_{1,26})=(&\sigma_1,\sigma_2^{-1},e,\sigma_1,\sigma_2,e,\sigma_2,e,\sigma_2,e,e,\sigma_2,\sigma_1,\sigma_2,e,\sigma_2,\nonumber\\
&\sigma_2^{-1},e,e,e,e,e,\sigma_2,\sigma_2,\sigma_2,\sigma_2^{-1},e).
\end{align}

\textbf{Example:} We consider the braid $\beta=\sigma_1^{12}\sigma_2^{12}\sigma_{1}^{-12}\sigma_2^{-12}$. Then
\begin{equation}
\rho_{3,1}(\beta)=\rho_{3,1}(\sigma_1)^{12}\rho_{3,1}(\sigma_2)^{12}\rho_{3,1}(\sigma_1)^{-12}\rho_{3,1}(\sigma_2)^{-12}=\text{id}.
\end{equation}
Some of the $\beta_{1,i}$ are trivial braids. For example, 
\begin{align}
\beta_{1,1}&=\prod_{k=0}^{11} A_{1,\left(\rho_{3,1}(\sigma_1)\right)^k(1)} \prod_{k=0}^{11}B_{1,\left(\rho_{3,1}(\sigma_2)\right)^k(1)}\prod_{k=1}^{12} A^{-1}_{1,\left(\rho_{3,1}(\sigma_1)\right)^{-k}(1)} \prod_{k=1}^{12}B^{-1}_{1,\left(\rho_{3,1}(\sigma_2)\right)^{-k}(1)}\nonumber\\
&=(A_{1,1}A_{1,2})^6(B_{1,1}B_{1,4}B_{1,7})^4(A_{1,2}^{-1}A_{1,1}^{-1})^6(B_{1,7}^{-1}B_{1,4}^{-1}B_{1,1}^{-1})^4\nonumber\\
&=\sigma_2^6(\sigma_2^{-1}\sigma_2)^4\sigma_{2}^{-6}(\sigma_2^{-1}\sigma_{2})^4\nonumber\\
&=e.
\end{align} 
We have
\begin{equation}
(\beta_{1,0},\beta_{1,1},\beta_{1,2},\ldots,\beta_{1,26})=(e,e,S,e,e,T,e,e,T,e,e,T,e,e,S,e,e,T,e,e,T,e,e,T,e,e,S),
\end{equation}
where $S=\sigma_2^6\sigma_1^6\sigma_2^{-6}\sigma_1^{-6}$ and $T=\sigma_1^6\sigma_2^6\sigma_1^{-6}\sigma_2^{-6}$. Note that here only the braids $\beta_{1,i}$ with $i\equiv 2\text{ mod }3$ are non-trivial.

Hence, like in the case of $n=2$, we witness a certain halving of even exponents. Note that both $\rho_{3,1}(S)$ and $\rho_{3,1}(T)$ are non-trivial, so that the action of $\beta$ on $27^2=729$ points is non-trivial. For example $27\times2+2$ is sent to $27\times26+2$.

Recall from Section \ref{sec:actions} that $h(B)=(B_{1,0},B_{1,1},B_{1,2},\ldots,B_{1,n^n-1},\sigma(B))$. Knowing $h(\sigma_1)$ and $h(\sigma_2)$ we can compute $h(B)$ for any braid word $B$ in a number of steps that grows linearly with the length of the braid word. Thus we can also compute $h(B_{j,i})$ in a linear time, which define $\rho_{3,j+1}(B)$. Note that the length of $B_{j,i}$ is at most the length of $B$ for all $j$, $i$. Therefore we can compute each $\rho_{3,j}(B)$, i.e., the action $\psi_{3}(\cdot,B)$ on the $j^{\text{th}}$ coordinate of $\mathbb{Z}_{n^n}$ in linear time with respect to the length of $B$. Note however that the number of lifted braids $B_{j,i}$ for a given $j$ is $(n^n)^j$ and thus grows exponentially with $j$. Therefore in particular actions of braids with a large number of strands $n$ are still quite expensive to compute.

The braids listed in Equations (\ref{eq:aas}) and (\ref{eq:bbs}) are formed by the roots of paths in the space of polynomials corresponding to the lifts of $A=\sigma_1$ and $B=\sigma_2$. Note that $A_{1,i}=A'_{1,i}$ and $B_{1,i}=B'_{1,i}$ as introduced in Section \ref{sec:seq}. We can perform the same computations for the other embedding of $Z_n$ into $V_n$, i.e., the braid associated to a path in $Z_n$ is given by the union of the 0-strand and the critical points of the corresponding polynomials. In the case of $A=\sigma_1$ and $B=\sigma_2$ these are:
\begin{align}
\label{eq:caas}
(\hat{A}_{1,0},\hat{A}_{1,1},\hat{A}_{1,2},\ldots,\hat{A}_{1,26})=(&e,e,e,e,e,e,\sigma_1^{-1},\sigma_1,\sigma_1,\sigma_2\sigma_1,\sigma_1,\sigma_1,\sigma_1,e,\nonumber\\
&\sigma_2,\sigma_1^{-1},e,e,\sigma_2,e,e,\sigma_1,e,e,e,e,\sigma_2)
\end{align}
and
\begin{align}
\label{eq:cbbs}
(\hat{B}_{1,0},\hat{B}_{1,1},\hat{B}_{1,2},\ldots,\hat{B}_{1,26})=(&e,e,e,\sigma_1,\sigma_2,e,e,e,e,e,e,\sigma_2,\sigma_1,\sigma_2\sigma_1,e,\sigma_2,\nonumber\\
&\sigma_2^{-1},e,\sigma_1^{-1},e,e,e,\sigma_1\sigma_2,\sigma_2\sigma_1,\sigma_2,\sigma_2^{-1},e).
\end{align}

We find that the braids that are formed by the critical points are sometimes, but not always the same as those formed by the roots. One difference is that in the case of the roots, the lifted braids $A_{1,j}$ and $B_{1,j}$ had at most the length of the original braid $A$ or $B$, while $\hat{A}_{1,9}$ for example has length 2. 

We can use the embedding of $Z_n$ into $V_n$ where a polynomial is mapped to its set of critical points to construct an action of $\mathbb{B}_n$ on $\mathbb{Z}_n$ completely analogous to the construction of $\psi_n$. By definition the action on $n^n$ points of this action is identical with that defined from $\psi_n$. For bigger values of $j$ however, the actions on $(n^n)^j$ points are different, since the braids in Equations (\ref{eq:caas}) and (\ref{eq:cbbs}) are different from those in (\ref{eq:aas}) and (\ref{eq:bbs}). 


\section{Dynamics}
\label{sec:dynamics}

The previous sections give rise to two types of maps, whose dynamics might be worth studying. One is produced by the braid group actions on $\mathbb{Z}_n$. The problem of describing this system includes for example questions about the orbits of an $n$-adic integer under the actions defined in Section \ref{sec:actions}. No action of the braid group on the $n$-adic integers can be transitive, since $\mathbb{B}_n$ is countable, while $\mathbb{Z}_n$ is uncountable, but we can prove that the corresponding action on $\mathbb{Z}/(n\times(n^{n-1})^j)\mathbb{Z}$ is transitive for all $j$.

The other problem that seems somewhat related to dynamical systems is concerning the image of a point $v\in V_n$ under repeated application of $\theta_n^{-1}$ and the image of a point $v\in V_n$ under repeated application of $\theta_n$. In this section we give a brief overview of the kind of questions that arise here.

\subsection{Transitivity}
\label{sec:trans}

We begin the discussion of transitivity of the group actions with a result on the topology of the covering spaces $Z_n^j$.
\begin{proposition}
\label{prop:trans}
$Z_n^j$ is path-connected for every $n$ and every $j$.
\end{proposition}
\begin{proof}
$V_n$ is clearly path-connected. We still show this explicitly in a way that is instructive for the proof of the path-connectedness of $Z_n^{j}$. Let $x$, $y$ be in $V_n$. Then in particular $x$ and $y$ are unordered $(n-1)$-tuples of non-zero complex numbers $(x_1,x_2,\ldots,x_{n-1})$ and $(y_1,y_2,\ldots,y_{n-1})$ in $\mathbb{C}^{n-1}/S_{n-1}$. Consider a corresponding ordered $(n-1)$-tuple, i.e., $\tilde{x}=(x_1,x_2,\ldots,x_{n-1})$ and $\tilde{y}=(y_1,y_2,\ldots,y_{n-1})$ as points in $\mathbb{C}^{n-1}$. Take a complex affine line $L\subset \mathbb{C}^{n-1}$ containing both $\tilde{x}$ and $\tilde{y}$.

The points in $\mathbb{C}^{n-1}$ that do not correspond to points in $V_n$ are the zeros of a finite number of polynomials, namely $(z_1,z_2,\ldots,z_{n-1})\mapsto z_i$ and $(z_1,z_2,\ldots,z_{n-1})\mapsto z_i-z_j$, $i\neq j$. Since neither $x$ nor $y$ are a zero of any of these functions, none of them are constant zero on $L$. Hence each of the holomorphic functions has a countable number of zeros on $L$ (in fact in this case a finite number) and the complement of the zeros in $L$ is path-connected. Thus there is a path from $x$ to $y$ in $V_n$ and $V_n$ is path-connected.

We proceed by induction on $j$, where we think of $V_n$ as $Z_n^0$. Let $x$, $y\in Z_n^j$ and assume that $Z_n^{j-1}$ is path-connected. Since $Z_n^j$ is a subset of $Z_n^{j-1}$, there is a path $\gamma$ from $x$ to $y$ in $Z_n^{j-1}$. Applying $\theta_n^{j-1}$ to it yields a path $\gamma'$ in $V_n$ with endpoints $\theta_n^{j-1}(x)$ and $\theta_n^{j-1}(y)$. We want to show that $\gamma'$ can be deformed to lie completely in $Z_n$.

The map that sends a monic polynomial with constant term equal to 0 to its (unordered) set of critical points is a homeomorphism $q$. 

We consider ordered $(n-1)$-tuples, $\widetilde{q(\theta_n^{j-1}(x))}$ and $\widetilde{q(\theta_n^{j-1}(y))}$ and the corresponding path $\widetilde{q(\gamma')}$ between them. We say a point in $\mathbb{C}^{n-1}$ corresponds to a point in $Z_n$ (respectively $V_n$) if it lies in $q(Z_n)$ (respectively $q(V_n)$) when considered as an unordered $(n-1)$-tuples. Again the sets of critical points that correspond to points in $V_n$, but not to points in $Z_n$ are the zeros of a finite number of holomorphic functions that are not constant zero. Therefore, the points that correspond to points in $q(Z_n)$ is dense and open in the set of points that correspond to points in $q(V_n)$. It follows that $\widetilde{q(\gamma')}$ can be deformed into a path that consists of finitely many line segments $\ell_i$ such that both endpoints of $\ell_i$ correspond to points in $q(Z_n)$. Now take complex affine lines $L_i\subset\mathbb{C}^{n-1}$. By the same arguments as in the case of $V_n$ each $\ell_i$ can be deformed to lie completely in $q(Z_n)$, i.e. each point on $\ell_i$ corresponds to a point in $q(Z_n)$. This gives a deformation of $\widetilde{q(\gamma')}$ to a path in $q(Z_n)$. This homotopy lifts to a homotopy of $\gamma'$ and ultimately to a homotopy of $\gamma$. Since applying $\theta_n^{j-1}$ to this new path results in a path in $Z_n$, applying $\theta_n^j$ to it results in a path in $V_n$. In other words, the constructed path between $x$ and $y$ lies in $Z_n^j$ and hence $Z_n^j$ is path-connected. 
\end{proof}

It follows from the definition of $\varphi_n$ and $\phi_n$ that they reduce to actions on $(n^{n-1})^j$ points and $n\times(n^{n-1})^j$ points respectively. These are given by the actions on the points in the fibres in $Z_n^j$. Theorem \ref{thm:main} \textit{ii)} is now simply a corollary of Proposition \ref{prop:trans}, since monodromy actions on the fibres in path-connected covering spaces are transitive.

\begin{corollary}
\label{cor:trans}
The action of $\mathbb{B}_n$ on $\mathbb{Z}/(n\times (n^{n-1})^j)\mathbb{Z}$ that comes from  the restriction of $\phi_n$ is transitive for all $n$ and $j$.
\end{corollary}

Note that the transitivity on each level does not contradict the lack of transitivity on their inverse limit $\mathbb{Z}/n\mathbb{Z}\times\mathbb{Z}_{n^{n-1}}\cong\mathbb{Z}_n$, since in general the orbit map does not commute with the inverse limit. Some conditions when they do commute are given in \cite{singh}, but they are not satisfied here.

With exactly the same methods as in Proposition \ref{prop:trans} and Corollary \ref{cor:trans} we can prove the analogous statements for the spaces $\hat{Z}_n^j$, defined from the embedding of $Z_n$ into $V_n$ that sends a polynomial to its critical points.

As we have seen in Section \ref{sec:comps} Corollary \ref{cor:trans} is not true for the corresponding actions on $\mathbb{Z}/(n^n)^j\mathbb{Z}$ that are restrictions of $\varphi_n$. It is not even true for $n=2$, $j=2$.

\subsection{Preimage and image sets}
\label{sec:prepost}

We study the image of a given point in $V_n$ under repeated application of $\theta_n^{-1}$. This is a sequence of unordered $n-1$-tuples of non-zero complex numbers (i.e., points in $V_n$), indexed by the $n$-adic integers. 

Let $\mathcal{Z}_x^j$ be the set of the complex numbers $z$ such that there exist distinct non-zero complex numbers $z_i\neq z$, $i=1,2,\ldots,n-2$ such that $j$ repeated applications of $\theta_n$ to $(z_1,z_2,\ldots,z_{n-2},z)$ result in $x\in V_n$. Let $\mathcal{Z}_x=\bigcup_j\mathcal{Z}_x^j$.
\begin{proposition}
The set $\{\arg(z):z\in\mathcal{Z}_x\}$ is dense in $S^1$ for all $x\in V_n$.
\end{proposition}
\begin{proof}
We have already mentioned a certain symmetry among the preimage points under $\theta_n$, namely if $(z_1,z_2,\ldots,z_{n-1})\in Z_n$ maps to $x\in V_n$, then $\xi^k(z_1,z_2,\ldots,z_{n-1})$ also maps to $x$ for all $k=0,1,2,\ldots,n-2$, where $\xi=\rme^{2\pi \rmi/n}$ is a $n^{\text{th}}$ root of unity and multiplication is defined componentwise. It follows by induction on $j$ that the set of complex numbers that are components of the preimage set of $j$ applications of $\theta_n$ of a point $x\in V_n$ is symmetric under rotation by $\tfrac{2\pi}{n^j}$. The proposition follows.
\end{proof}

What can we say about the modulus of the elements of $\mathcal{Z}_x$? Again the case of $n=2$ is simple enough to be done by hand. Recall that for $n=2$ we have $Z_n^j=V_n=\mathbb{C}\backslash\{0\}$. The modulus of the non-zero root $z$ of a monic polynomial $u(u-z)$ is $2\sqrt{|x|}$, where $x$ is the critical value of $u(u-z)$. Therefore the moduli of elements in $\mathcal{Z}_x^j$ all converge to $4$ as $j$ goes to infinity for all $x\in V_n=\mathbb{C}\backslash\{0\}$. Figure \ref{fig:n2j10} shows the distribution of elements of $\bigcup_{j=1}^k\mathcal{Z}_x^j$ for $x=1$ and $k=4,7,10$.

\begin{figure}
\centering
\includegraphics[height=3.5cm]{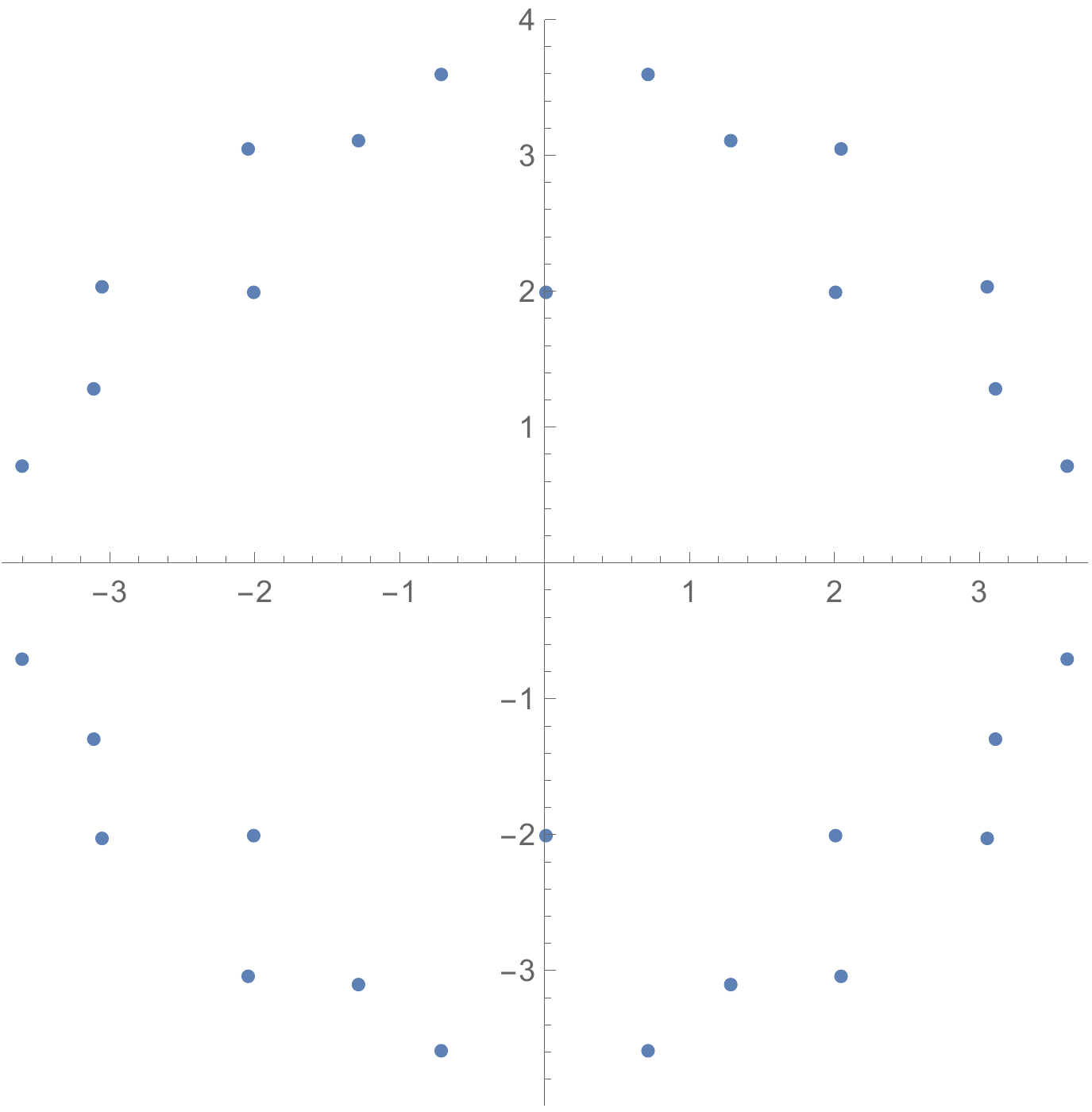}\quad
\includegraphics[height=3.5cm]{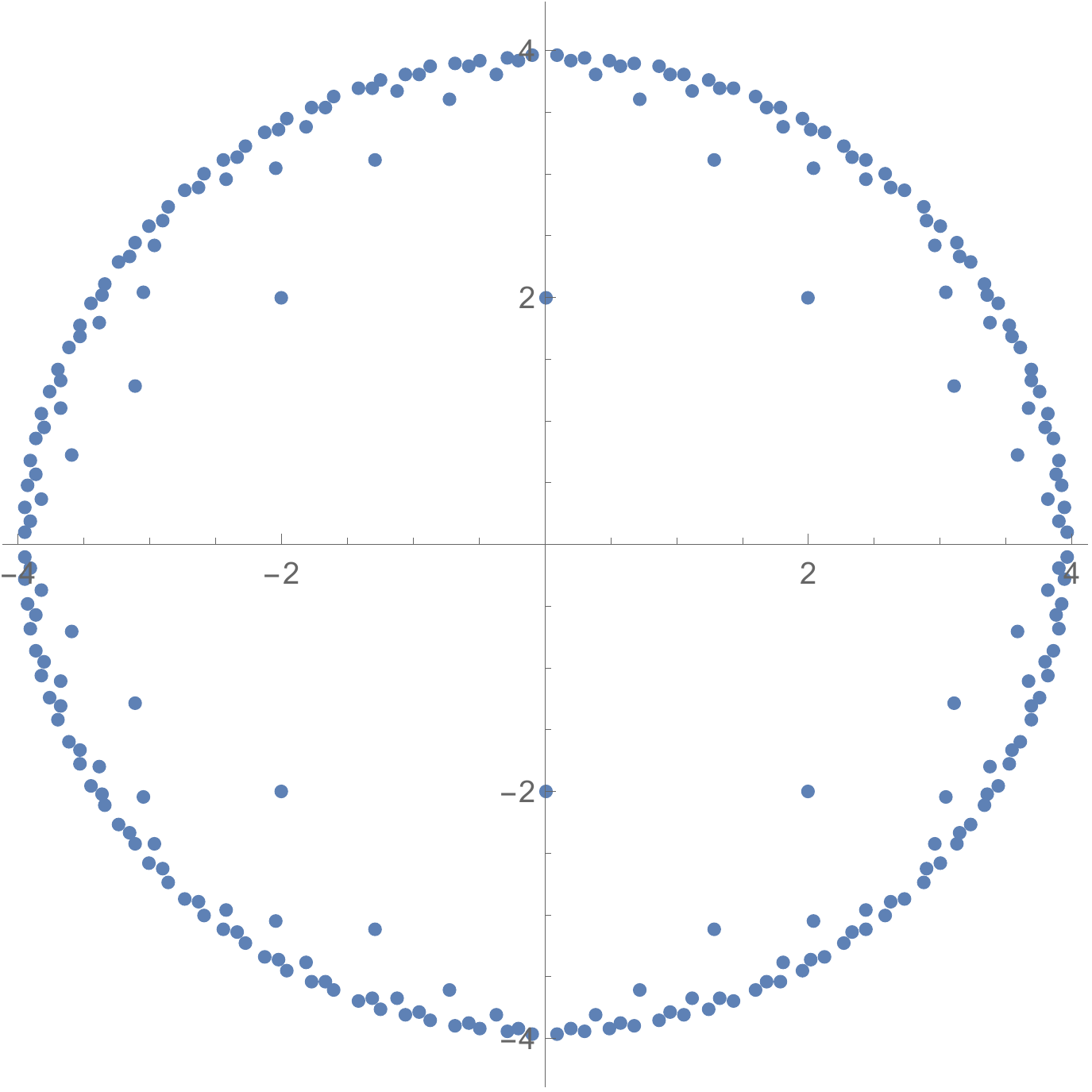}
\includegraphics[height=3.5cm]{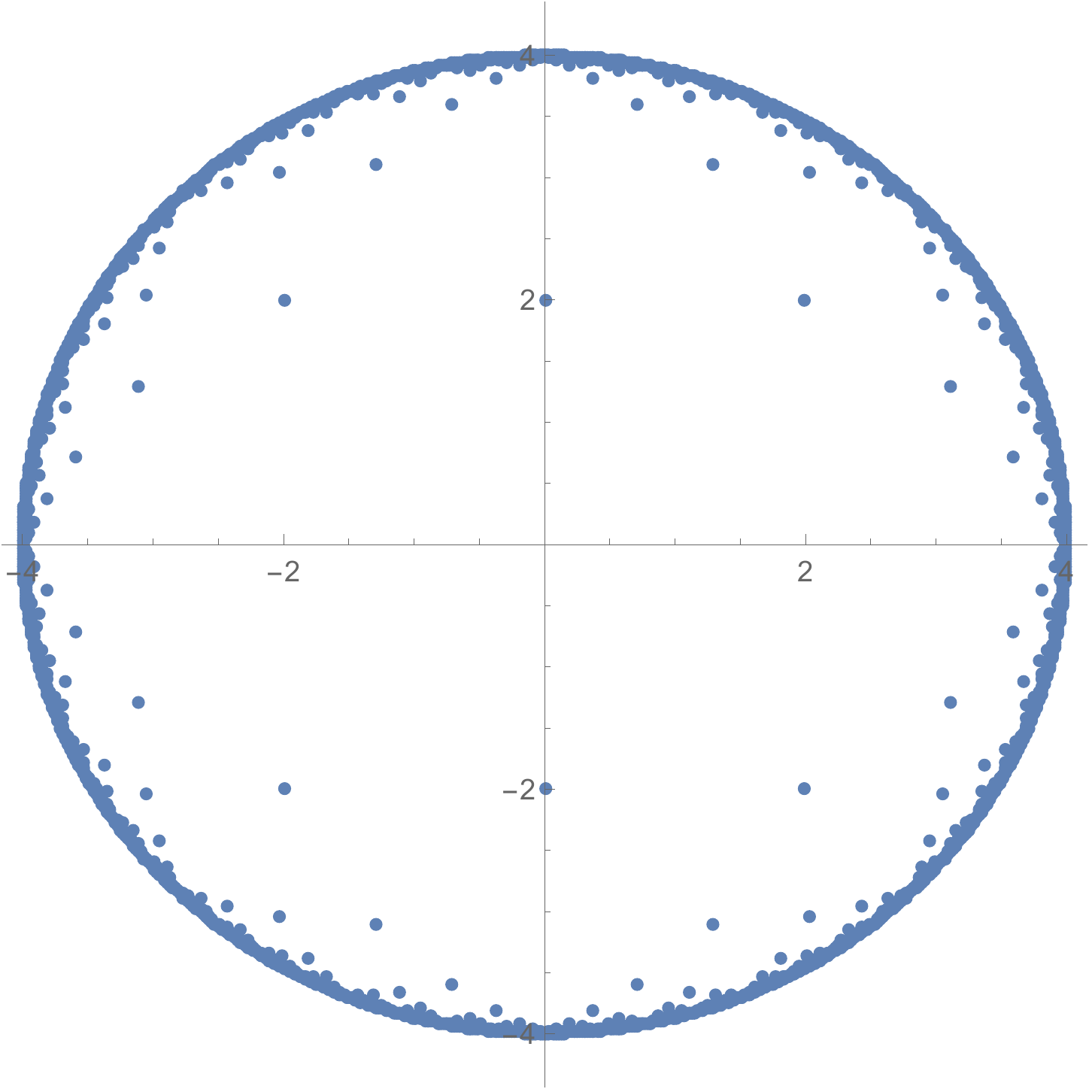}
\caption{\label{fig:n2j10}The sets $\bigcup_{j=1}^k\mathcal{Z}_{1}^j$ for $k=4$, $k=7$ and $k=10$.}
\end{figure}

For the case of $n=3$ we plot $\bigcup_{j=1}^k\mathcal{Z}_{(\rme^{\rmi \pi/3}/2,-1)}^j$ for $k=3,4$ in Figure \ref{fig:n3j5}. The base point $(\rme^{\rmi \pi/3}/2,-1)$ is chosen arbitrarily. We can see that the moduli of points in $\bigcup_{j=1}^k\mathcal{Z}_{( \rme^{\rmi \pi/3}/2,-1)}^j$ seem to lie in a finite interval.

\begin{figure}
\centering
\includegraphics[height=3.5cm]{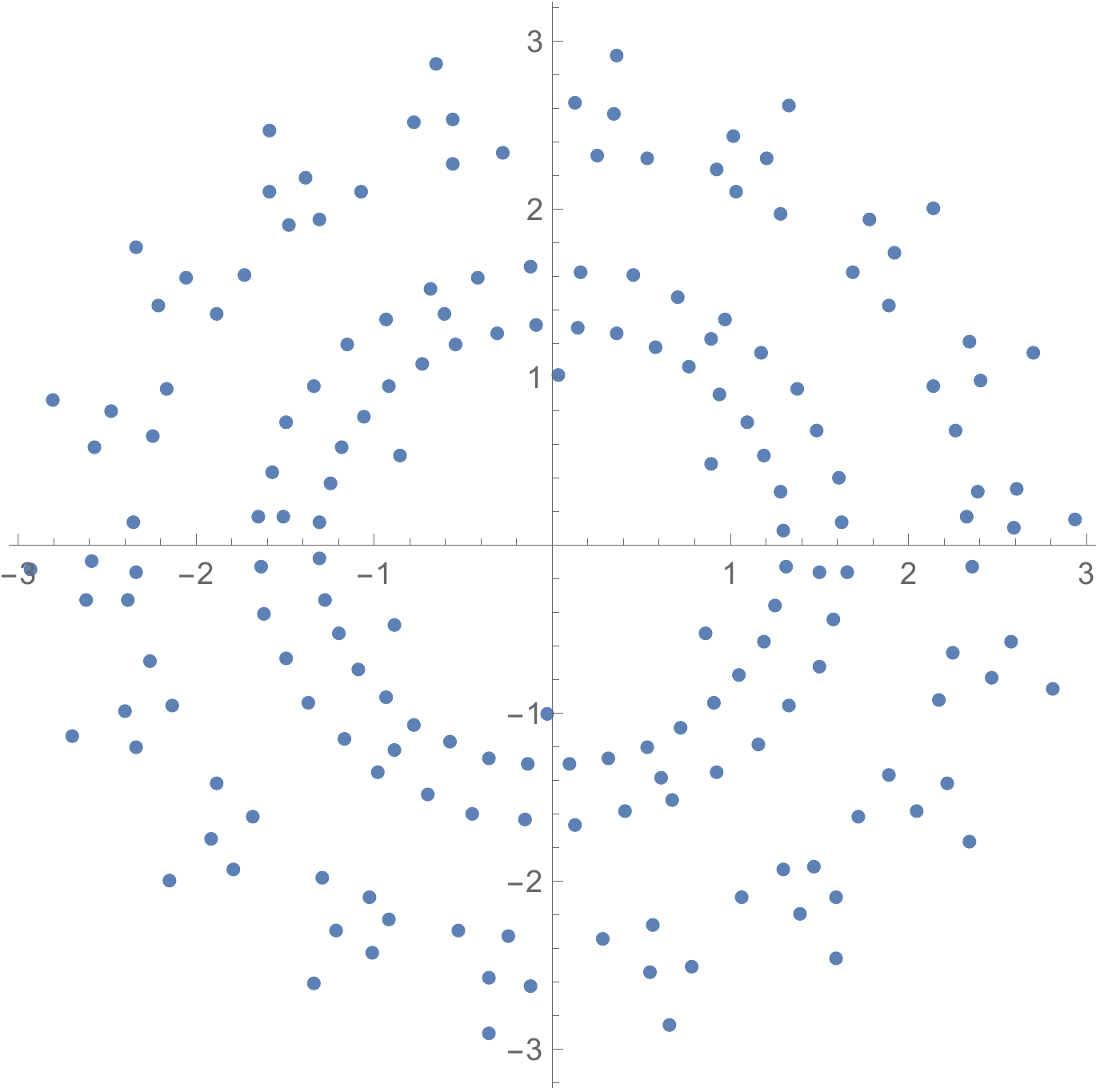}\quad
\includegraphics[height=3.5cm]{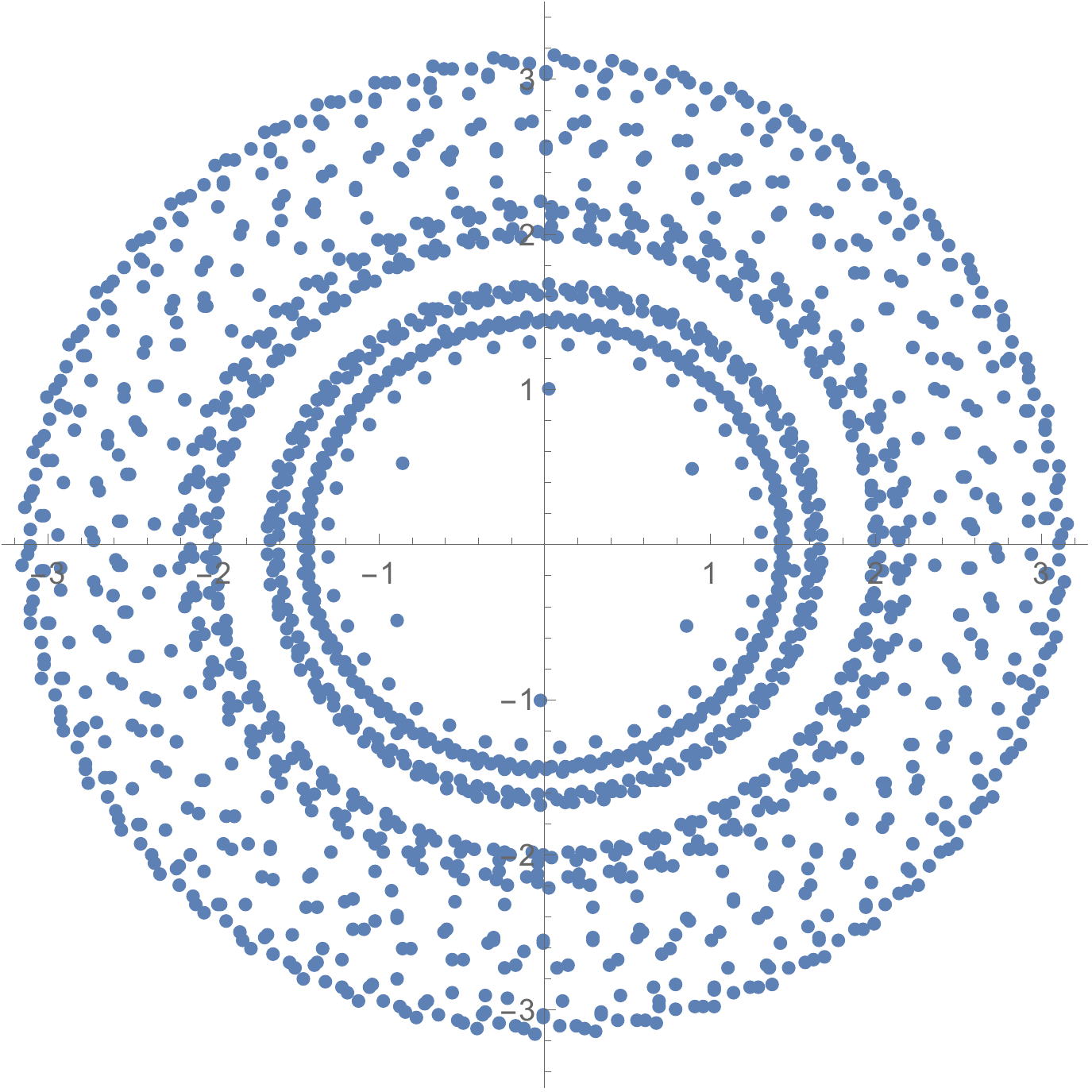}
\caption{\label{fig:n3j5}The sets $\bigcup_{j=1}^k\mathcal{Z}_{(\rme^{\rmi \pi/3}/2,-1)}^j$ for $k=3$ and $k=4$.}
\end{figure}

The sets $\mathcal{Z}_x$ offer insight in the complex numbers that arise under repeated application of $\theta_n^{-1}$ to $x$. We could just as well repeatedly apply $\theta_n$ to $x$. There are different cases to consider. If $x\in Z_n^j$, then the first $j$ applications of $\theta_n$ will result in points that are in $V_n$. Therefore, if $x\in\bigcap_j Z_n^j$, then no matter how many times we apply $\theta_n$, we always end up in $V_n$. If we take for example any tuple that consists of one positive real number $x_1$ and one negative real number $x_2$. Then the corresponding polynomial $u\prod_{i=1}^2(u-x_i)$ has again one maximum with positive real value and one minimum with negative real value. Hence, no matter how many times we apply $\theta_3$, we always obtain a set of critical values that are non-zero and disjoint and thus in $V_3$.

In contrast, if $x\in Z_n^j\backslash Z_n^{j+1}$, then the $j+1^{\text{th}}$ application of $\theta_n$ results in a point $y$ that is not in $V_n$. This means it gives an unordered $(n-1)$-tuple of non-zero complex numbers, where at least two of them are equal. We do not have to stop here though. The map $\theta_n$ is still well-defined as the map that sends a polynomial $u\prod_{i=1}^{n-1}(u-z_i)$ (or equivalently, an unordered $(n-1)$-tuple $(z_1,z_2,\ldots,z_{n-1})$) to its set of critical values. It is just not a covering map anymore. Since at least two components of $y$ are equal, the corresponding polynomial has a double root. Hence $\theta_n(y)$ is an unordered $(n-1)$-tuple of complex number with at least one of them equal to 0. The polynomial corresponding to this tuple now has a double root at $0$, so its image under $\theta_n$ also contains 0. Therefore, from now on every additional application of $\theta_n$ will result in a tuple that contains 0.

We can study the sequence $a_j\in\mathbb{N}$ of number of zeros in $\theta_n^j(x)$. Note that it is monotone increasing and obviously bounded above by $n-1$ for every $x\in\mathbb{C}^{n-1}$. Therefore, it must be constant after a while and there is a well-defined limit value $a=\lim_{j\to\infty} a_j$. In fact, the polynomial $p\equiv 0$ is the only polynomial with $(0,0,\ldots,0)$ as its set of critical values. Therefore, the sequence is actually bounded by $n-2$.

\begin{remark}
For points like in the example above ($n=3$, one positive, one negative entry), the sequence $a_j$ is constant 0, but for all points that are not in $\bigcap_j Z_n^j$ the limit is at least 1. In the case of $n=3$ this implies that the limit is 0 for $\bigcap_j Z_3^j$ and 1 in all other cases.
\end{remark}
Again, the discussion in this section is focussed on the relation between the roots and the critical values of polynomials. We could easily proceed analogously to study the relation between critical points and critical values. We take a $(n-1)$-tuple of complex numbers $c=(c_1,c_2,\ldots,c_{n-1})$ and identify it with the polynomial $f(u)=\int_0^u \prod_{i=1}^{n-1}(w-c_i)\rmd w$. We then send $c$ to the set of critical values of $f$. If $f$ is in $Z_n$, then the image of this map is in $V_n$. In any case we obtain a new $(n-1)$-tuple, which we can use to iterate the process. In this case, it is not the number of zeros that is increasing, but the number of pairs of identical entries. 

\section{Normal subgroups}
\label{sec:normal}

The actions $\phi_n$ and $\psi_n$ on $\mathbb{Z}_n$ each provide us with a sequence of homomorphisms $\mathbb{B}_n\to S_{n\times(n^{n-1})^j}$ and $\mathbb{B}_n\to S_{(n^n)^j}$, respectively. These correspond to the restricted actions on $n\times(n^{n-1})^j$ and $(n^n)^j$ points respectively. Therefore, we have two descending series of normal subgroups $N_j$ and $H_j$ of $\mathbb{B}_n$ corresponding to the kernels of these homomorphisms. The fact that they are descending sequences, i.e., $N_{j+1}\subseteq N_j$ and $H_{j+1}\subseteq H_j$ for all $j\in\mathbb{N}$ follows from the compatibility of the homomorphisms for different $j$, which follows from the definition of the actions.

In the case of both $\phi_n$ and $\psi_n$ the action on $n^n$ points (i.e., $j=1$) is compatible with the permutation representation $\pi:\mathbb{B}_n\to S_n$. This means that the normal subgroups $N_j$ and $H_j$ are in fact normal subgroups of the pure braid group.
    

    
With the exception of $\psi_2$ (cf. Section \ref{sec:comps}) it is not known if the presented actions $\phi_n$ and $\psi_n$ are faithful. This is the case if and only if the intersections of the normal subgroups $\bigcap_{j=1}^{\infty}N_j$ and $\bigcap_{j=1}^{\infty}H_j$, respectively, are the trivial braid $\{e\}$. Here we show that the descending sequences do not stabilise, i.e., there is no $M\in\mathbb{N}$ such that $N_j=N_M$ (or $H_j=H_M$) for all $j\geq M$.

\begin{lemma}
\label{stabilize}
The descending sequence $N_j$ of $\mathbb{B}_n$ does not stabilise for any $n$.
\end{lemma}
\begin{proof}
Since $Z_n^j$ is path-connected for all $n$ and $j$ (cf. Proposition \ref{prop:trans}), the image of $\mathbb{B}_n$ in $S_{n\times(n^{n-1})^j}$ consists of at least $n\times(n^{n-1})^j$ permutations, which is a lower bound for the index of $N_j$. Since every $N_j$ has finite index, this shows that the sequence cannot stabilise.
\end{proof}



In order to prove that the normal subgroups $H_j$ do not stabilise either, we need several lemmas.

\begin{lemma}
\label{pathequal}
Let $x\in Z_n$ with 0 in $i^{\text{th}}$ position and $y\in Z_n$ with 0 in $j^{\text{th}}$ position. Let $B\in\mathbb{B}_n$ be a braid with $\pi_B(i)=j$. Then there is a path from $x$ to $y$ in $Z_n$ that corresponds to a parametrisation of $B$.  
\end{lemma}

\begin{proof}
This proof is a variation of the proof of Proposition \ref{prop:trans}. Let $x=(x_1,x_2,\ldots,x_{n-1})$, $y=(y_1,y_2,\ldots,y_{n-1})\in Z_n$. If $\pi_B(i)=j$, there is a path $\gamma'(t)=(\gamma_1'(t),\gamma_2'(t),\ldots,\gamma_{n-1}'(t))$, $t\in[0,1]$ from $x$ to $y$ in $\mathbb{C}^{n-1}/S_{n-1}$ that is a parametrisation of $B$.
Recall that in the proof of Proposition \ref{prop:trans} we showed the path-connectedness of $Z_n$ by showing that a path $\gamma$  in $Z_n^0=V_n$ from $x$ to $y$ can be deformed to a path in $Z_n$. In $V_n$ the path $\gamma$ can be taking to be any braid $B$ as long as it satisfies $\pi_B(i)=j$ and the deformation does not change the braid type. The lemma follows.   
\end{proof}

\begin{lemma}
\label{apple}
Let $n>2$. There is a base point $v\in V_n$ that has 0 in the $\lfloor\frac{n+1}{2}\rfloor^{\text{th}}$ position and that has two preimage points $x,y\in Z_n$ that have 0 in the $\lfloor\frac{n+1}{2}\rfloor^{\text{th}}$ position. 
\end{lemma}

\begin{proof}
Consider an $n-1$-tuple $x$ of real numbers with 0 in $\lfloor\frac{n+1}{2}\rfloor^{\text{th}}$ position. Note that $v=\theta_n(x)$ is also an $n-1$-tuple of real numbers with 0 in $\lfloor\frac{n+1}{2}\rfloor^{\text{th}}$ position. Recall that the preimage set of a point in $V_n$ under $\theta_n$ is invariant under multiplication by $\rme^{2\pi\rmi /n}$. Multiplying $x$ by $\rme^{2\pi\rmi /n}$ thus results in another preimage point of $v$, say $y$. If $n\neq2,4$, then $y$ also has $0$ in $\lfloor\frac{n+1}{2}\rfloor^{\text{th}}$ position and $y$ is different from $x$ because it does not consist of real numbers. For $n=4$ we can apply the same argument to an $n-1$-tuples $x$ of complex numbers that have an argument of $\pm \epsilon$, for some small $\epsilon>0$, and $0$ in $\lfloor\frac{n+1}{2}\rfloor^{\text{th}}$ position.
\end{proof}

\begin{proposition}
\label{prop:hj}
The sequence $H_j$ does not stabilise.
\end{proposition}
\begin{proof}
Consider the points $x$, $y$ in $Z_n$ given by Lemma \ref{apple}. Since $Z_n$ is path-connected, there is a path from $x$ to $x$ and a path from $x$ to $y$. Applying $\theta_n$ to these paths results in two loops in $V_n$ with base point $v=\theta_n(x)$. The two loops correspond to two braids, say $B_1$ and $B_2$, and $\pi_{B_j}(\lfloor\tfrac{n+1}{2}\rfloor)=\lfloor\tfrac{n+1}{2}\rfloor$ for $j=1,2$. Note that $B_1$ and $B_2$ induce different permutations on $\mathbb{Z}/n^n\mathbb{Z}$. One of them fixes $x$, while the other maps $x$ to $y$.

Since $x$ and $y$ are in $Z_n$ and since they both have 0 in $\lfloor\frac{n+1}{2}\rfloor^{\text{th}}$ position, there is a path in $Z_n$ from $x$ to $y$ that corresponds to a parametrisation of $B_j$,$j=1,2$. The same is true for the existence of a path from $x$ to $x$ corresponding to $B_j$, $j=1,2$. Thus in total we have four paths in $Z_n$, two corresponding to $B_1$ and two to $B_2$. By Lemma \ref{pathequal} there are four paths in $Z_n$, two from $x$ to $x$ and two from $x$ to $y$, corresponding to parametrisations of $B_1$ and $B_2$. Applying $\theta_n$ to these four paths gives four loops in $V_n$ based at $v$, corresponding to four different braids, whose permutation representation fixes $\lfloor\frac{n+1}{2}\rfloor$. Note that by construction all of these four braids induce different permutations on $\mathbb{Z}/(n^n)^2\mathbb{Z}$. This process can be iterated an arbitrary number of times and we obtain $2^j$ different braids that induce different permutations on $\mathbb{Z}/(n^n)^j\mathbb{Z}$. In other words, the index of $H_j$ is bounded below by $2^j$. Exactly like in the proof of Lemma \ref{stabilize} this implies that the sequence $H_j$ does not stabilise.    
\end{proof}

This concludes the proof of Theorem \ref{thm:main} \textit{iii)}.

There are other examples of descending sequences of normal subgroups of $\mathbb{B}_n$ that do not stabilise, for example the congruence subgroups, whose intersection is the Torelli group. It remains to be seen if the results from this section can be improved to prove (or disprove) the faithfulness of the constructed actions.


Again, the results from this section remain true, when one considers the alternative embedding of $Z_n$ into $V_n$, given by the map that sends a polynomial to its critical points. 

\section{Constructions of real algebraic links}
\label{sec:realalg}




We can use the computations from Section \ref{sec:comps} to construct \textit{real algebraic links} in $S^3$. We use this term in the sense of Perron \cite{perron} as the real analogues of Milnor's algebraic links, links of isolated critical points of polynomials $f:\mathbb{R}^4\to\mathbb{R}^2$. This should not be confused with knotted algebraic varieties in $\mathbb{RP}^3$ as they were introduced by Viro \cite{viro}, which are also called real algebraic links.

\begin{definition}
\label{def:ral}
A link $L$ is real algebraic if there exists a polynomial $p:\mathbb{R}^4\to\mathbb{R}^2$ such that
\begin{itemize}
\item $p$ has an isolated singularity at the origin, i.e., $p(0)=0$, $\nabla p(0)=0$ and there is a neighbourhood $U$ of $0$ such that $0$ is the only point in $U$ where the rank of $\nabla p$ is not full,
\item $p^{-1}(0)\cap S_{\rho}^3=L$ for all small enough radii $\rho$. 
\end{itemize}
\end{definition} 

The number $0$ in Definition \ref{def:ral} refers to the origin in $\mathbb{R}^4$ and $\mathbb{R}^2$ and the zero matrix of size 2-by-4, respectively.

Milnor showed that all real algebraic links are fibred \cite{milnor}, but in contrast to the algebraic links the real algebraic links are not classified yet. Benedetti and Shiota conjectured that all fibred links are real algebraic links \cite{benedetti}. So far however, the set of links that are known to be real algebraic is still comparatively small. 
\begin{remark}
\label{rem:list}
To our knowledge the following list covers all links that are known to be real algebraic.
\begin{itemize}
\item All algebraic links, i.e. links of isolated singularities of complex polynomials: These are certain iterated cables of torus links.
\item If $f, g:\mathbb{C}^2\to\mathbb{C}$ are complex polynomials with isolated singularities at the origin and $L_{f\cup g}=fg^{-1}(0)\cap S^3_{\rho}$ is fibred, then $f\overline{g}$ has an isolated singularity with $L_{f\cup g}=L_{f\cup \overline{g}}$ as the link of the singularity \cite{pichon}.

\item \textit{Odd} fibred links: These are links that can be parametrised in $S^3$ such that they and the fibration map are invariant under the antipodal map on $S^3\subset\mathbb{R}^4$; $i(x_1,x_2,x_3,x_4)=(-x_1,-x_2,-x_3,-x_4)$ \cite{looijenga}. This includes $K\# K$ if $K$ is a fibred knot.

\item Closures of squares of homogeneous braids \cite{bode:real} (cf. Definition \ref{def:homo}). This includes the figure-eight knot, which has been shown to be real algebraic by Perron \cite{perron} and Rudolph \cite{rudolph:isolated} before.  
\end{itemize}
\end{remark}

\begin{definition}
\label{def:homo}
A braid $B$ on $s$ strands is called homogeneous if for every $i=1,2,\ldots,s-1$ the generator $\sigma_i$ appears in the word $B$ if and only if $\sigma_i^{-1}$ does not appear.
\end{definition}
We give a brief summary of the proof in \cite{bode:real} and explain how the computations from the previous sections allow us to generalise this result.

The proof in \cite{bode:real} can be summarised as follows. 
Let $B$ be a braid on $n$ strands. For every component $C$ of the closure of $B$ let $F_C$, $G_C:[0,2\pi]\to\mathbb{R}$ be trigonometric polynomials, such that
\begin{equation}
\label{eq:gpara}
\bigcup_C\bigcup_{j=1}^{n_C}\left(F_C\left(\frac{t+2\pi j}{n_C}\right),G_C\left(\frac{t+2\pi j}{n_C}\right),t\right),\quad t\in[0,2\pi]
\end{equation}
is a parametrisation of $B$, where $n_C$ denotes the number of strands in the component $C$ of the closure of $B$. 
We now define
\begin{equation}
\label{eq:defg}
g_\lambda:\mathbb{C}\times[0,2\pi]\to \mathbb{C},\quad g_\lambda(u,t)=\prod_{C}\prod_{j=1}^{n_C}\left(u-\lambda F_C\left(\frac{t+2\pi j}{n_C}\right)-\rmi \lambda G_C\left(\frac{t+2\pi j}{n_C}\right)\right).
\end{equation}
Note that the roots of $g_\lambda$ are precisely the parametrisation of $B$ given by Equation (\ref{eq:gpara}) scaled by the factor $\lambda$.

We then define $p_{\lambda,k}:\mathbb{C}^2\to\mathbb{C}$
\begin{align}
\label{eq:polysing}
\tilde{p}_{\lambda,k}((u,r\rme^{\rmi t}))&=r^{2nk}g_{\lambda}\left(\frac{u}{r^{2k}},t\right),\\
\tilde{p}_{\lambda,k}((u,0))&=u^n
\end{align}
where $n$ is the number of strands of $B$ and $k$ a sufficiently large integer. Then $\tilde{p}_{\lambda,k}$ is a polynomial in $u$, $v$, $\overline{v}$ and $\sqrt{v\overline{v}}$. Changing the variable from $t$ to $2t$ results in 
\begin{align}
\label{eq:polysing}
p_{\lambda,k}((u,r\rme^{\rmi t}))&=r^{2nk}g_{\lambda}\left(\frac{u}{r^{2k}},2t\right),\\
p_{\lambda,k}((u,0))&=u^n,
\end{align}
which can be written as a complex polynomial in $u$, $v$ and $\overline{v}$ and is thus a polynomial map $\mathbb{R}^4\to\mathbb{R}^2$. If $\lambda$ is chosen sufficiently small, the vanishing set of $p_{\lambda,k}$ intersects each 3-sphere $S^3_{\rho}$ of radius $\rho\leq1$ in the closure of $B^2$.

Furthermore, $p_{\lambda,k}$ has an isolated singularity if and only if $g_{\lambda}$ has no argument-critical points, which is equivalent to the non-vanishing of $\tfrac{\rmd \arg(v_i(t))}{\rmd t}$ for all $t\in[0,2\pi]$ and all $i=1,2,\ldots,n-1$, where $v_j(t)$, $i=1,2,\ldots,n-1$ are the critical values of the complex polynomial $g_{\lambda}(u,t)$. This inequality has the nice geometric interpretation that for all $j$ the critical value $v_j$ always twists around 0 with the same orientation, either always clockwise or always counterclockwise. We showed in \cite{bode:2016lemniscate, bode:2016polynomial} that homogeneous braids can be parametrised as the roots of $g_{\lambda}$ such that the condition $\tfrac{\rmd \arg(v_i(t))}{\rmd t}\neq 0$ for all $t\in[0,2\pi]$ and all $i=1,2,\ldots,n-1$ is satisfied. For a more recent and complete treatment of the idea for this proof we point the reader to\cite{bode:thesis, bode:braids}. It follows that closures of squares of homogeneous braids are real algebraic. 



If we want to generalise this construction, it suffices to find other braids $B$ that can be parametrised as in Equation (\ref{eq:gpara}) such that the resulting polynomial $g_\lambda$ in Equation (\ref{eq:defg}) has no argument-critical points, i.e., $\arg g_\lambda$ is a fibration. Then the closure of $B^2$ is again real algebraic. This can be summarised in the context of the covering map $\theta_n:Z_n\to V_n$ as follows.
\begin{proposition}
\label{prop:summary}
Let $v(t)=(v_1(t),v_2(t),\ldots,v_{n-1}(t))$, $t\in[0,2\pi]$, be a loop in $V_n$ such that $\tfrac{\partial \arg(v_i(t))}{\partial t}\neq0$ for all $i=1,2,\ldots,n-1$ and all $t\in[0,2\pi]$ and such that one of the lifts of $v(t)$ is a loop $\tilde{\gamma}$ in $Z_n$. We denote the braid that is traced out by the roots of $\tilde{\gamma}$ by $B$. Then the closure of $B^2$ is real algebraic.
\end{proposition}

We now focus on the case of $n=3$ and use the computations from Section \ref{sec:comps} to prove Theorem \ref{thm:real}.

\begin{proof}[Proof of Theorem \ref{thm:real}]
The computations of Section \ref{sec:comps} tell us all lifts of any 3-strand braid. We consider the lifts of the following affine braids, all of which can parametrised as $(0,v_1(t),v_2(t),\ldots,v_{n-1}(t))$ such that for all $i$ we have $\tfrac{\partial\arg (v_i(t))}{\partial t}>0$ for all $t$ or $\tfrac{\partial \arg(v_i(t))}{\partial t}=0$ for all $t$:
\begin{itemize}
\item $\beta_1=\sigma_1^2$,
\item $\beta_2=\sigma_2^{-1}\sigma_1^2\sigma_2$,
\item $\beta_3=\sigma_1^2\sigma_2$,
\item $\beta_4=\sigma_1^2\sigma_2^2$,
\item $\beta_5=\sigma_2\sigma_1^2\sigma_2$.
\end{itemize} 
Figure \ref{fig:figfig} depicts a parametrisation of the desired form for each of these braids.

We focus on the lifts of each $\beta_i$ that starts at $z_1$ (cf. Section \ref{sec:comps}). In general, this is not a loop in $Z_3$, but from the previous computations we can easily find the smallest power of each $\beta_i$ for which it is. We thus obtain the following five braids whose lift that starts at $z_1$ is a loop:
\begin{itemize}
\item $\beta_{1}=\sigma_1^2$,
\item $\beta_{2}^2=(\sigma_2^{-1}\sigma_1^2\sigma_2)^2$,
\item $\beta_3^6=(\sigma_1^2\sigma_2)^6$,
\item $\beta_4^6=(\sigma_1^2\sigma_2^2)^6$,
\item $\beta_5^3=(\sigma_2\sigma_1^2\sigma_2)^3$.
\end{itemize}

We now concatenate these braids in such a way that every non-zero strand moves at some point, i.e., for all $i$ we have $\tfrac{\partial v_i}{\partial t}(h)\neq0$ for some $h$. This means that we either need to use $\beta_3^6$ (because it moves both non-zero strands) or at least one of $\beta_1$ and $\beta_4^6$ (which move one strand) and at least one of $\beta_2^2$ and $\beta_5^3$ (which move the other strand). Every braid that results from such a concatenation can be deformed slightly so that the parametrisation satisfies the conditions on $\gamma$ in Proposition \ref{prop:summary}. Therefore the closure of the square of the braid that corresponds to the lift at $z_1$ is real algebraic.

The relevant lifts are given by:
\begin{itemize}
\item $w_1=\beta_{1_{1,1}}=\sigma_2$,
\item $w_2=(\beta_2^2)_{1,1}=\sigma_1^2$,
\item $w_3=(\beta_3^6)_{1,1}=(\sigma_1\sigma_2\sigma_1)^2$,
\item $w_4=(\beta_4^6)_{1,1}=(\sigma_2\sigma_1\sigma_2^{-1}\sigma_1\sigma_2)^2$,
\item $w_5=(\beta_5^3)_{1,1}=\sigma_2^{-1}\sigma_1\sigma_2^2\sigma_1$.
\end{itemize}

This means that the square of any concatenation of the $w_i$s, $B=\prod_{j=1}^{\ell}w_{i_j}$ (subject to the extra condition on the $i_j$s), closes to a real algebraic link by Proposition \ref{prop:summary}. The condition on the $i_j$s in the statement of the theorem is a consequence of the earlier remark on the requirement that every non-zero strand of the braid of critical values has to move at some point.
The equivalent statement for $\epsilon=-1$ follows from the same argument applied to the inverses of the $\beta_i$s.\end{proof}



\begin{figure}
\centering
\labellist
\large
\pinlabel a) at 0 270
\pinlabel b) at 0 -30
\pinlabel c) at 0 -320
\pinlabel d) at 0 -650
\pinlabel e) at 0 -970
\endlabellist
\includegraphics[height=3.5cm]{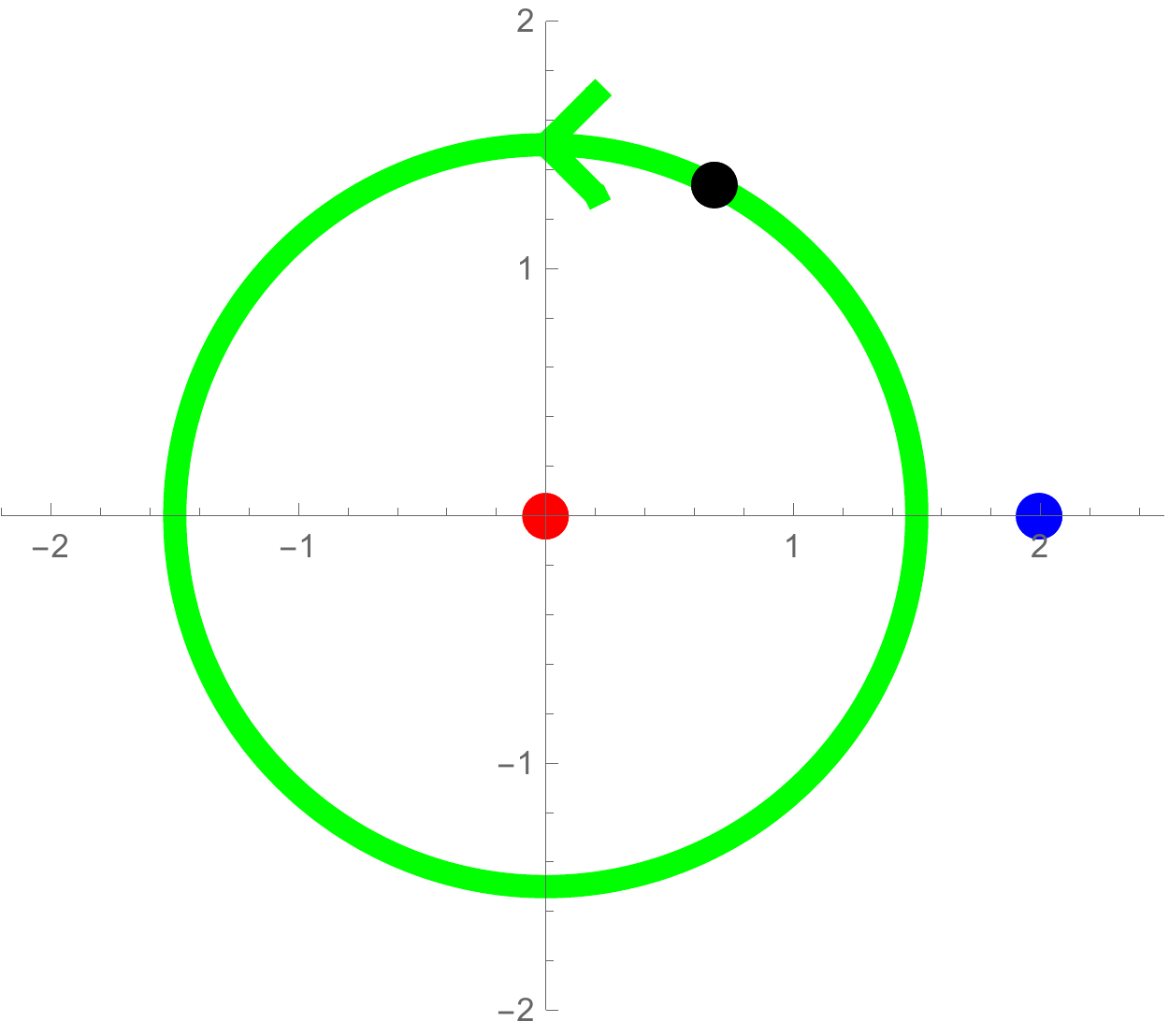}\qquad
\includegraphics[height=3.5cm]{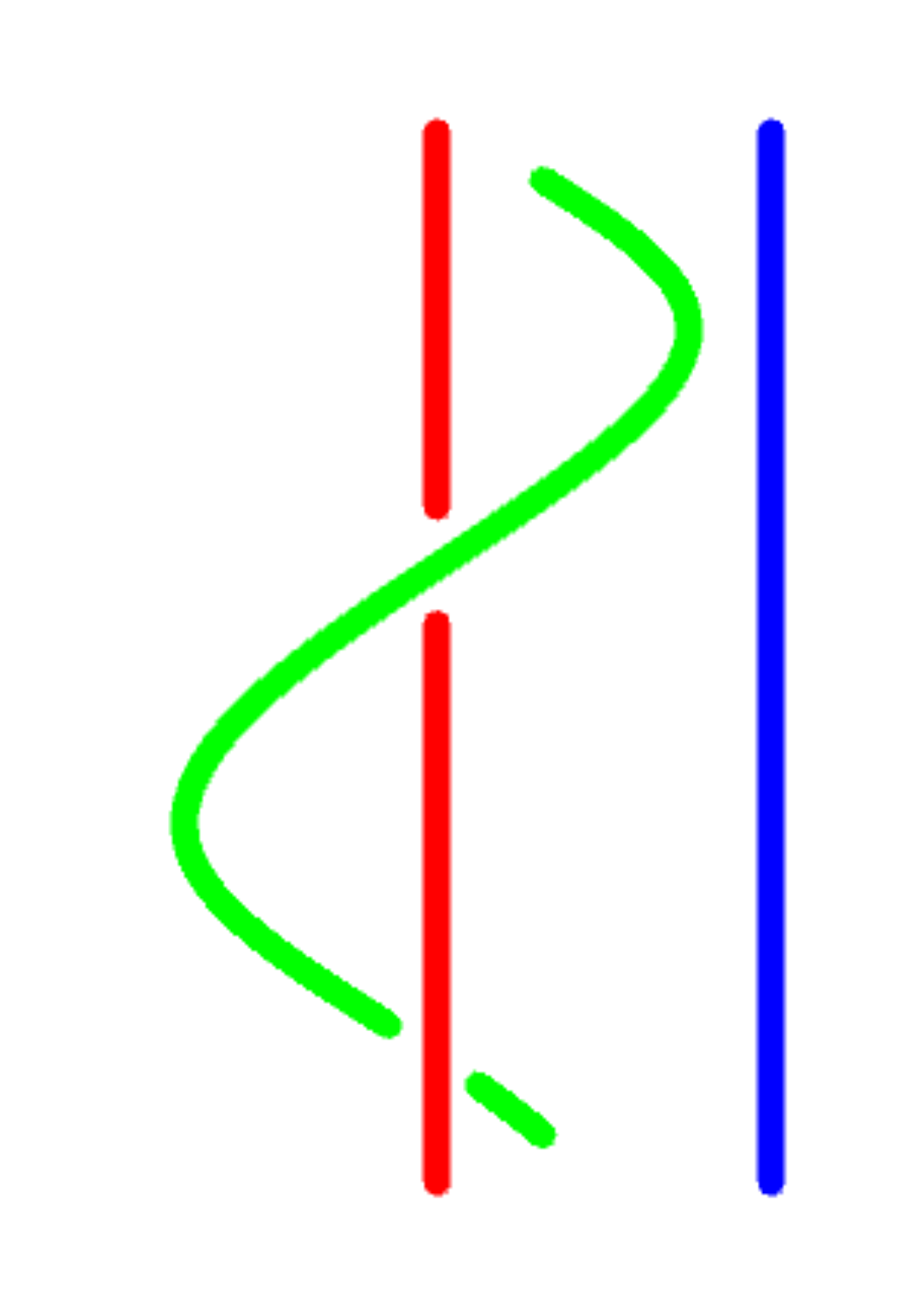}\\
\includegraphics[height=3.5cm]{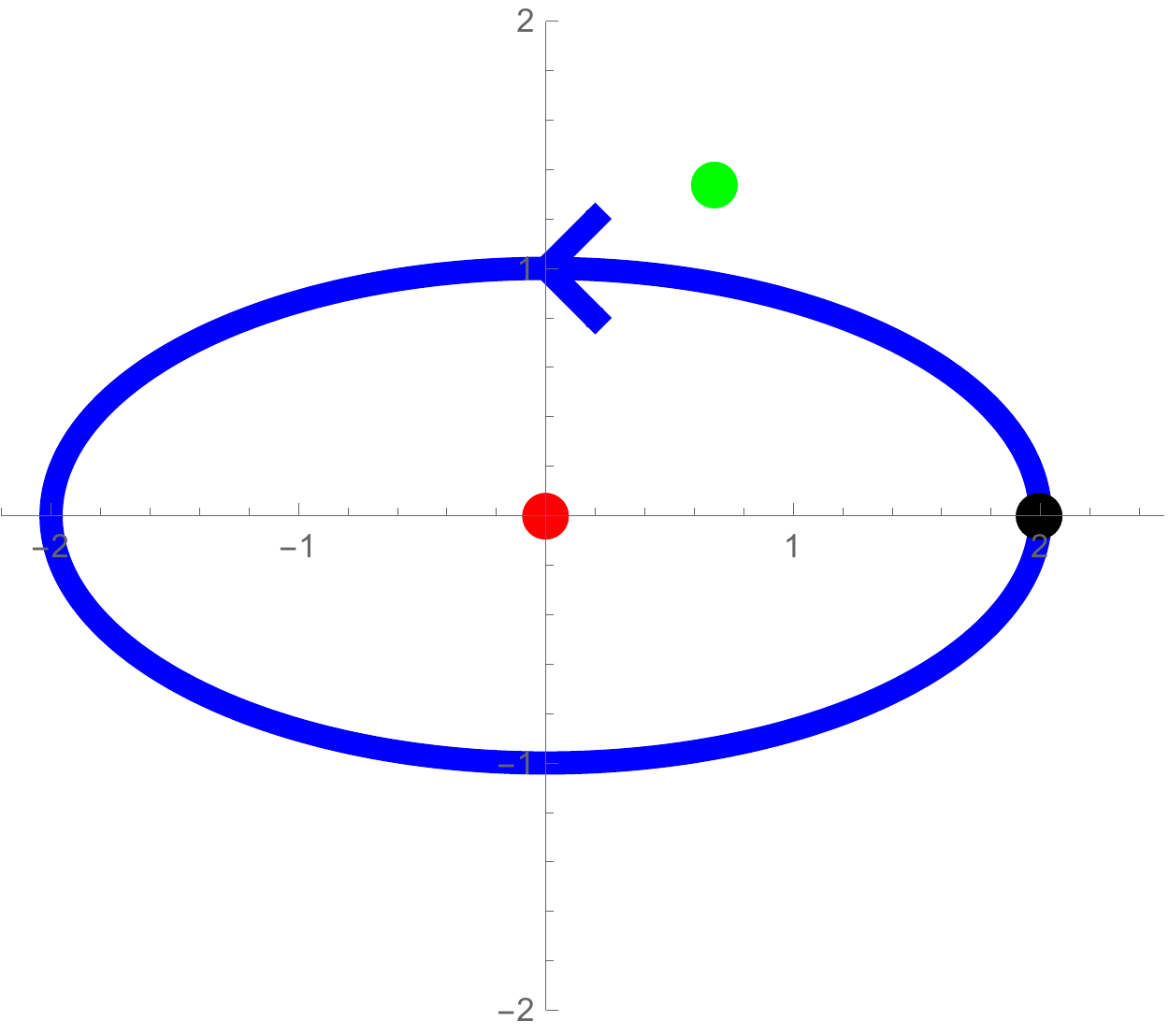}\qquad
\includegraphics[height=3.5cm]{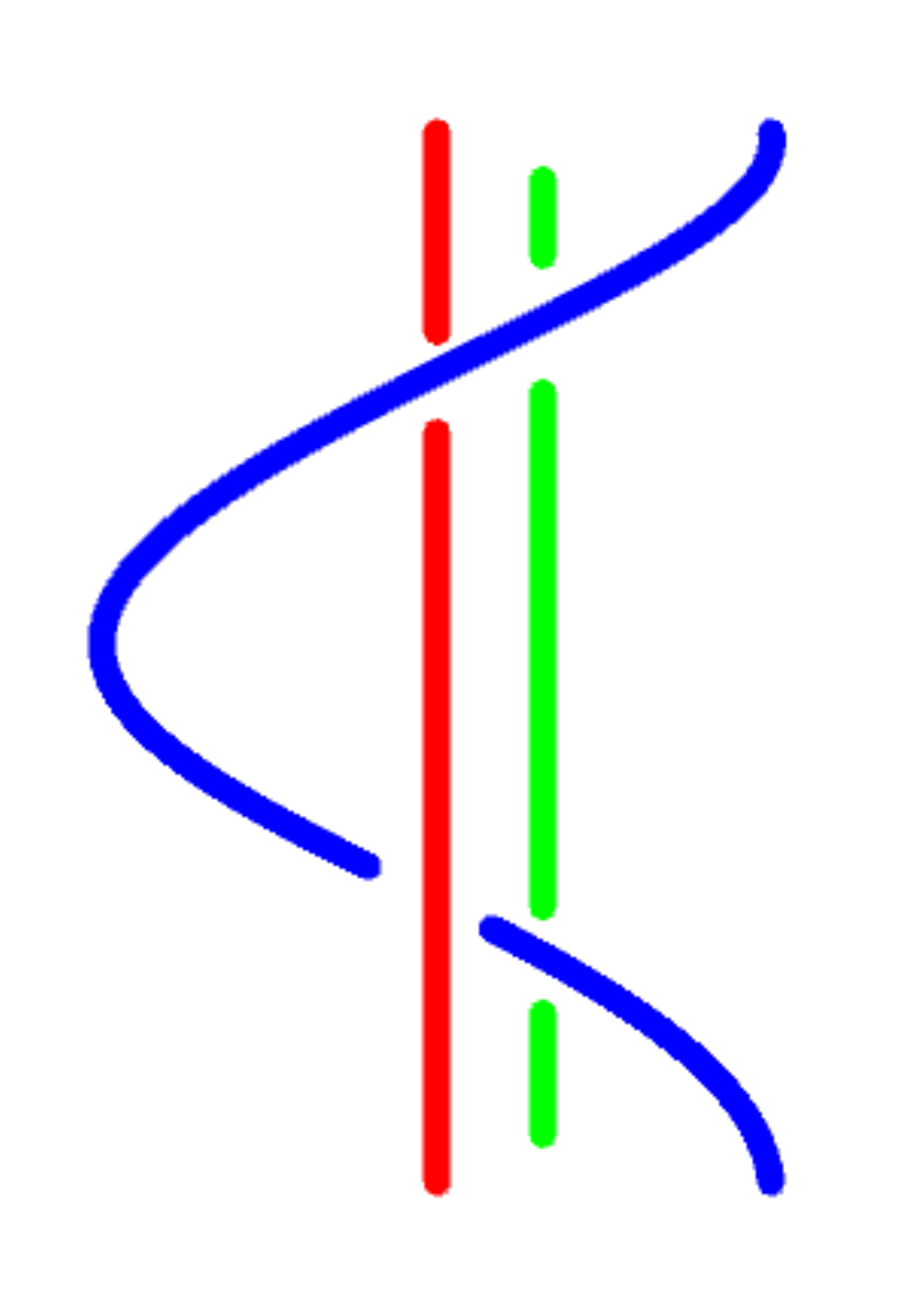}\\
\includegraphics[height=3.5cm]{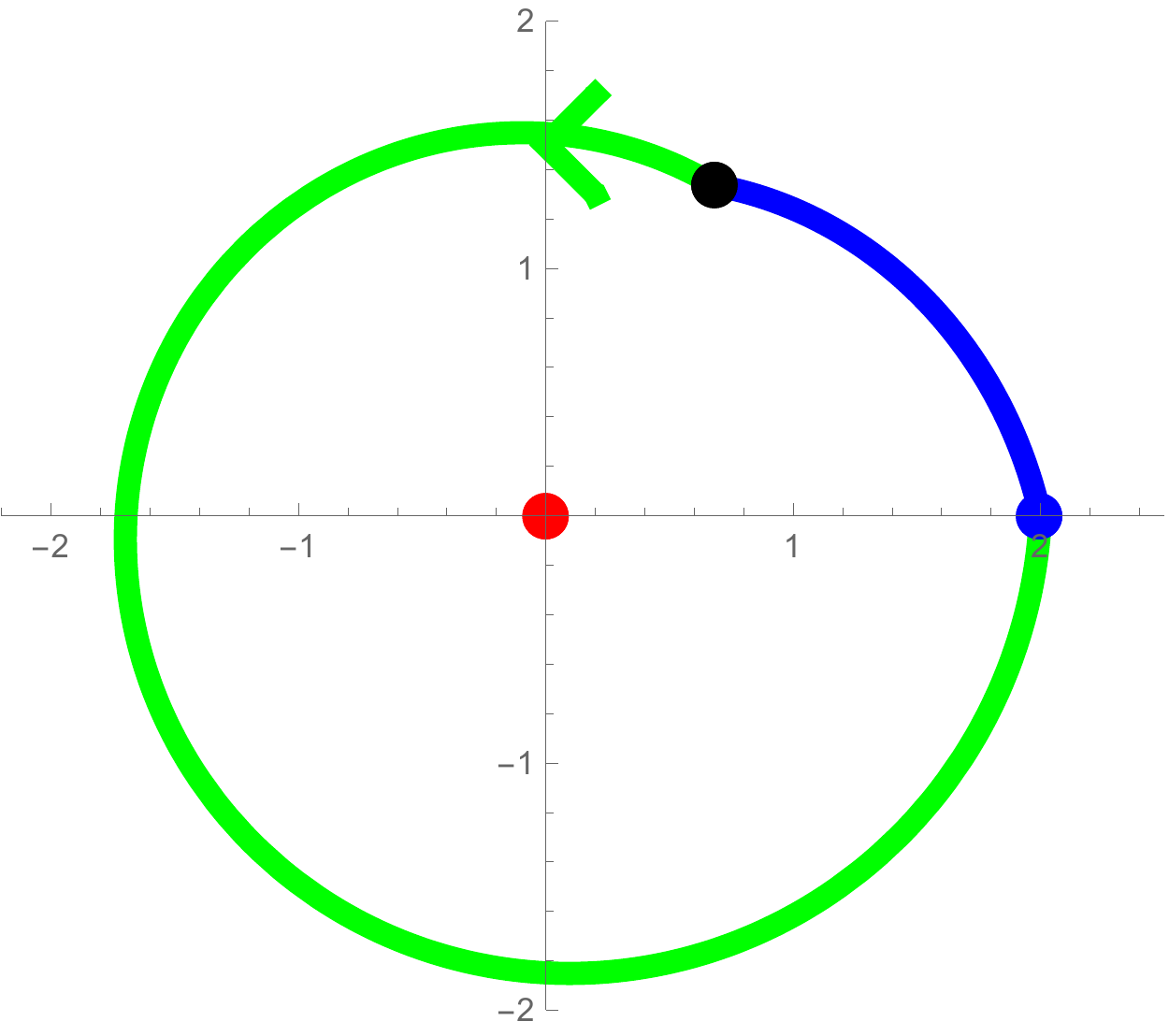}\qquad
\includegraphics[height=3.5cm]{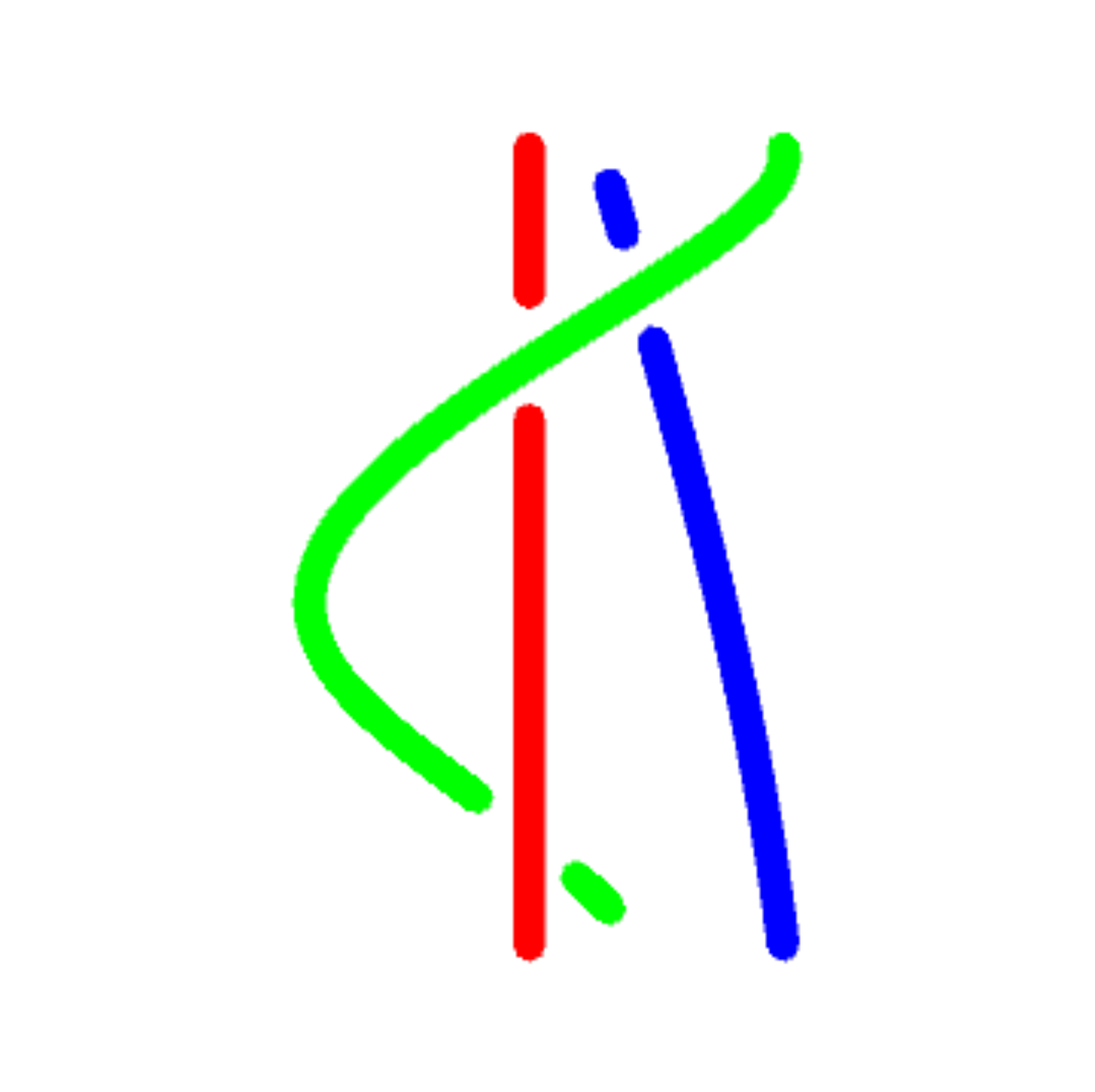}\\
\includegraphics[height=3.5cm]{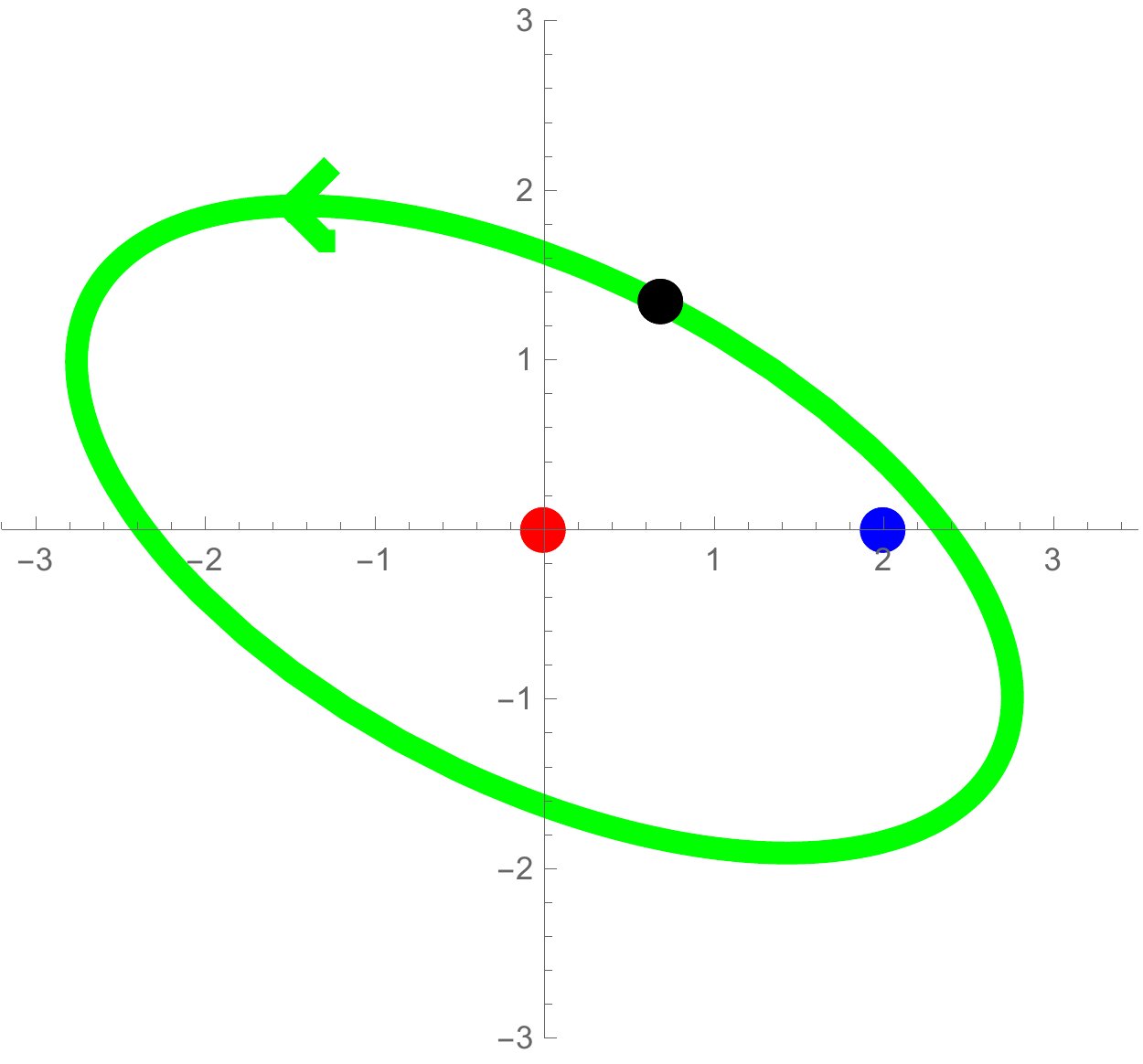}\qquad
\includegraphics[height=3.5cm]{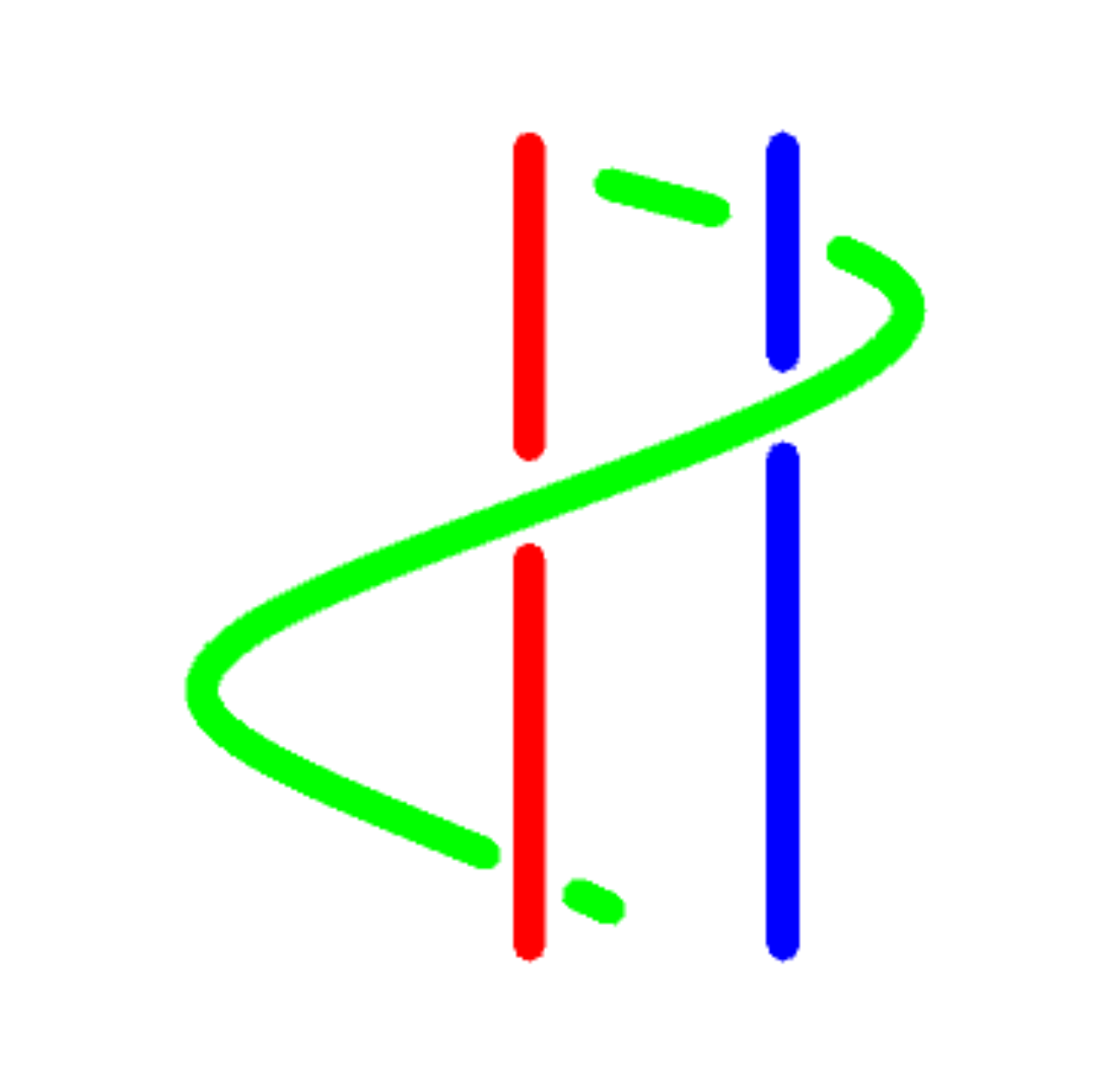}\\
\includegraphics[height=3.5cm]{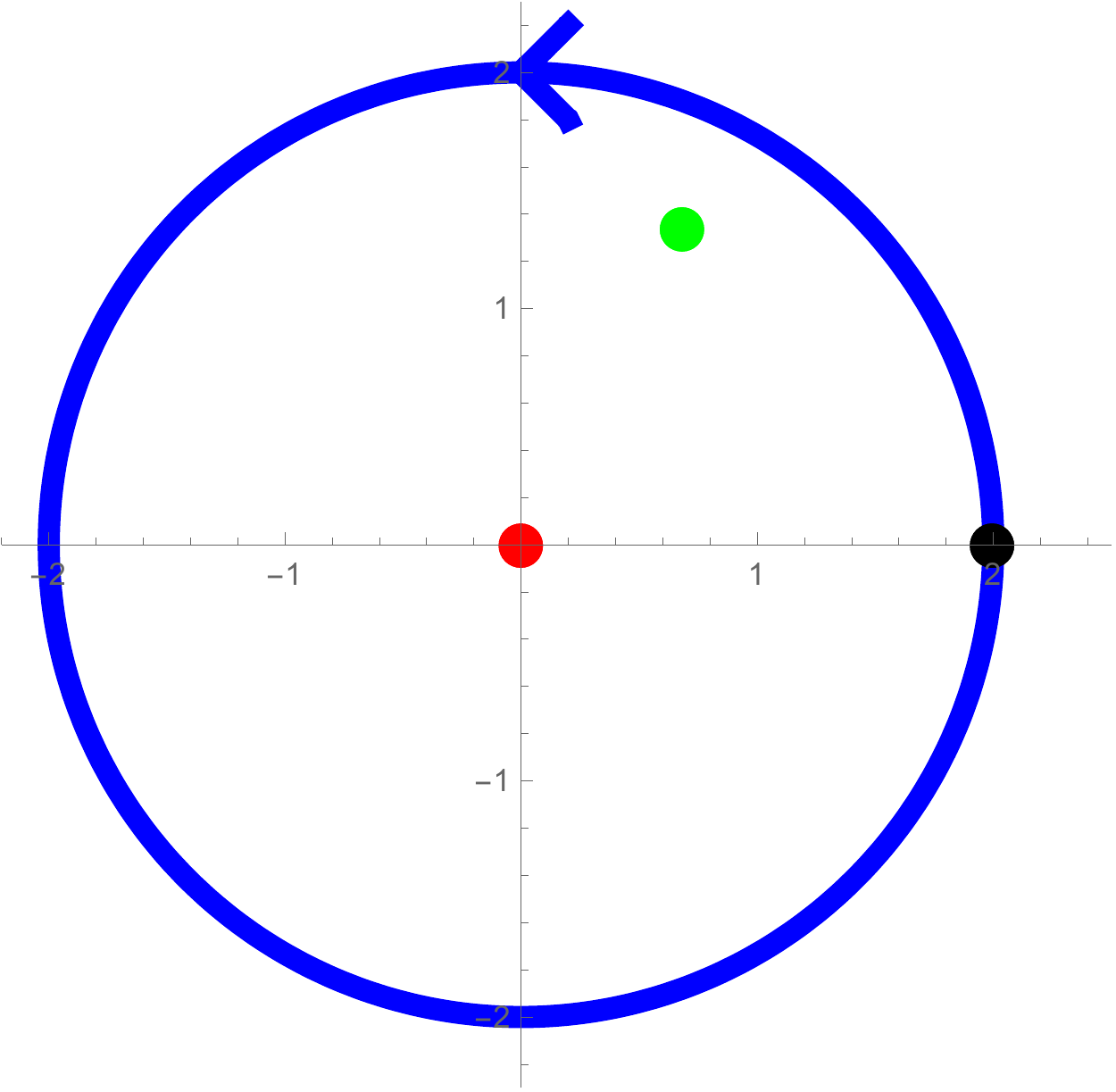}\qquad
\includegraphics[height=3.5cm]{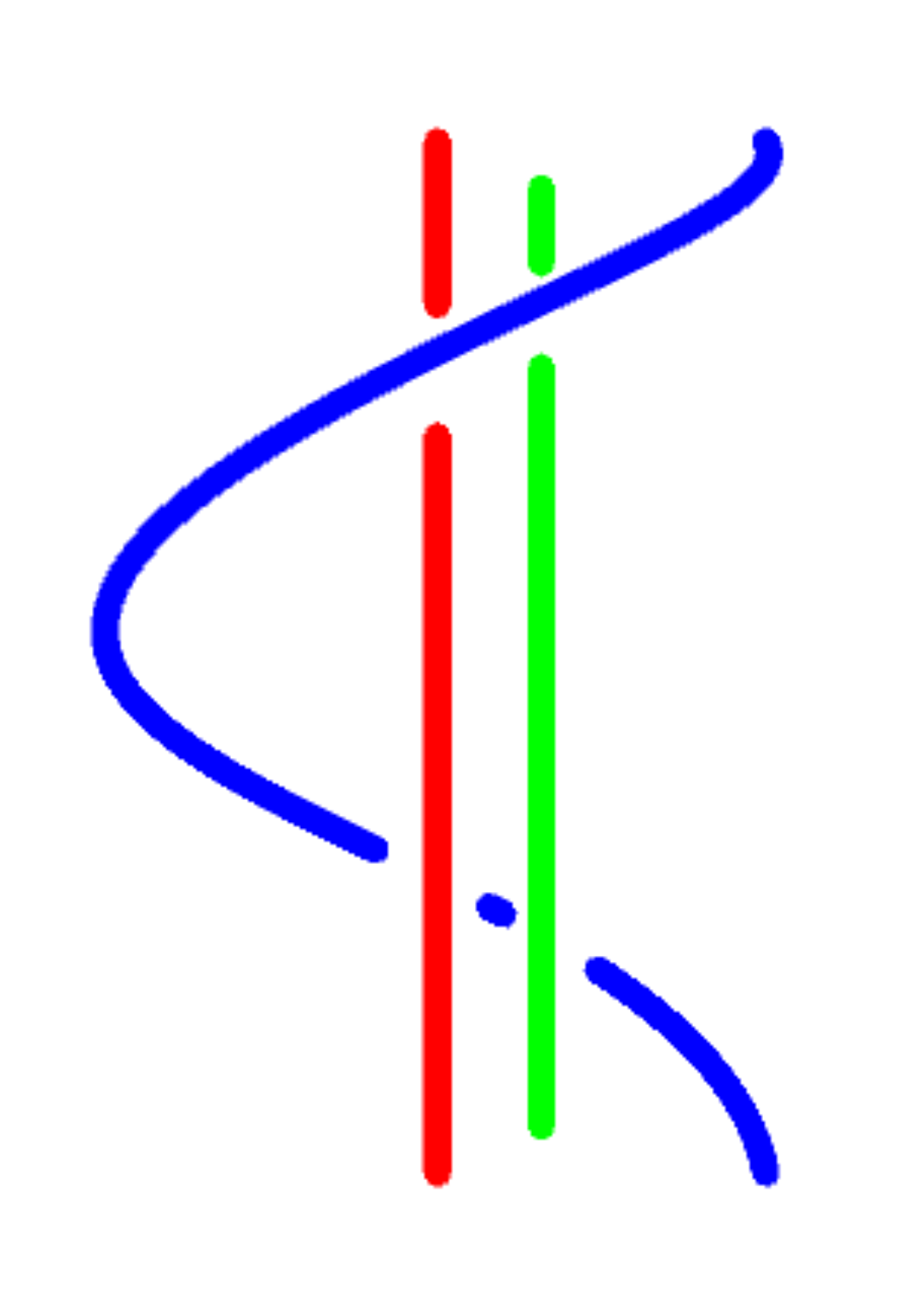}
\caption{Braid parametrisations of the $\beta_i$s depicted as movements of points in the complex plane and as the projection on the $xz$-plane resulting in braid diagrams. a) $\beta_1$. b) $\beta_2$. c) $\beta_3$. d) $\beta_4$. e) $\beta_5$. \label{fig:figfig}}
\end{figure}

Theorem \ref{thm:real} only considers the special case where $\tfrac{\partial \arg v_i(t)}{\partial t}$ has the same sign for all $i$ and $t$, while in general it suffices to have $\tfrac{\partial \arg v_i(t)}{\partial t}\neq0$, i.e., we could have that $\text{sign}\left(\tfrac{\partial \arg v_i(t)}{\partial t}\right)=-\text{sign}\left(\tfrac{\partial \arg v_j(t)}{\partial t}\right)$ for different $i$ and $j$. It also only considers the lifts that start at $z_1$. There should be 26 other similar families of real algebraic links corresponding to the lifts that start at the other $z_i$s.


A more systematic study of which affine braids have a parametrisation without turning points and that have a loop as one of their lifts is an ongoing project.

\subsection{The closure of $((w_5^{-2}w_1^{-1})^2w_2^{-1})^2$}
\begin{theorem}
\label{thm:newnew}
Let $B=(w_5^{-2}w_1^{-1})^2w_2^{-1})^2=((\sigma_1^{-1}\sigma_2^{-2}\sigma_1^{-1}\sigma_2\sigma_1^{-1}\sigma_2^{-2}\sigma_1^{-1})^2\sigma_1^{-2}$. Then the closure of $B^2$ is real algebraic, but is not in any of the families that have already been known to be real algebraic (as listed in Remark \ref{rem:list}). 
\end{theorem}

We establish this result through a sequence of lemmas.
\begin{lemma}
The closure $L$ of the braid $((\sigma_1^{-1}\sigma_2^{-2}\sigma_1^{-1}\sigma_2\sigma_1^{-1}\sigma_2^{-2}\sigma_1^{-1})^2\sigma_1^{-2})^2$ is not quasipositive.
\end{lemma}
\begin{proof}
Suppose that $L$ were quasipositive. The braid index of $L$ is 3, since the braid above has three strands and $L$ consists of 3 components. Then by \cite{hayden} it is the closure of a quasipositive braid on three strands. Since $L$ has 3 components, this braid must be a pure braid. The sum of the linking numbers of the components is $-10-8+2=-16$, which is independent of the link diagram representing $L$. For a pure braid the sum of the linking numbers is half of the exponent sum of the braid, which for a quasipositive braid is positive. Hence $L$ is not quasipositive.  
\end{proof}
It follows that $L$ is not algebraic, since algebraic links are quasipositive. Furthermore, since for complex polynomials $f, g$ with isolated singularities the link of the singularity of $f\overline{g}$ is isotopic to $L_{fg}=fg^{-1}(0)\cap S_{\epsilon}^3$, links of this form are transverse $\mathbb{C}$-links and hence quasipositive too. Thus $L$ cannot be shown to be real algebraic using Pichon's result.

\begin{lemma}
The closure $L$ of the braid $((\sigma_1^{-1}\sigma_2^{-2}\sigma_1^{-1}\sigma_2\sigma_1^{-1}\sigma_2^{-2}\sigma_1^{-1})^2\sigma_1^{-2})^2$ is not odd.
\end{lemma}  
\begin{proof}
Suppose that $L$ were invariant under the antipodal map 
\begin{equation}
i:(x_1,x_2,x_3,x_4)\mapsto(-x_1,-x_2,-x_3,-x_4).
\end{equation} 
Since the pairwise linking numbers are different for each pair of components, the map $i$ has to fix the three components of $L$ setwise. Odd links are \textit{freely periodic} links of period 2 (cf. \cite{kawauchi} chapter 10.2). Hartley \cite{hartley} studied properties of the Alexander polynomial of freely periodic knots and his result generalizes to the multivariably Alexander polynomial $\Delta_1$ of links if the periodic map fixes all components setwise as is the case here. For a freely periodic 3-component link with period 2 states that there exists a polynomial in $\mathbb{Z}[t_1^{\pm1/2},t_2^{\pm1/2},t_3^{\pm1/2}]$ such that up to multiplication by units
\begin{equation}
\Delta_1(L)(t_1 t_2 t_3, t_2,t_3)=f(t_1,t_2,t_3)\times f(-t_1,t_2,t_3)
\end{equation}
and analogous statements for the other components \cite{hillman}. Computing $\Delta_1(L)$ using the KnotTheory package \cite{katlas} in Mathematica and factorising it shows that this condition is not satisfied. Hence $L$ is not freely periodic and therefore not odd.
\end{proof}
This shows that $L$ cannot shown to be real algebraic using Looijenga's construction.

\begin{lemma}
\label{lem:degcon}
Let $B$ be a homogeneous braid and $L=\cup_{i=1}^k L_i$ be its closure, where each $L_i$ is a knot. Then
\begin{equation}
\deg\nabla(L)\geq 2\sum_{i<j}|\text{lk}(L_i,L_j)|-k+1,
\end{equation} 
where $\nabla(L)$ is the Conway polynomial of $L$ and $\text{lk}(L_i,L_j)$ is the linking number of $L_i$ with $L_j$.
\end{lemma}
\begin{proof}
By Bell \cite{bell} it is $m=\deg\nabla (L)+n-1$, where $m$ is the length of the homogeneous braid word $B$ and $n$ is its number of strands and also $n\leq\deg\nabla (L)+1$. The number of crossings between strands from differenct components $L_i$, $L_j$, $i\neq j$ is at least $2|\sum_{i<j}\text{lk}(L_i,L_j)|$ and the number of crossings of strands from the same component is at least $n-k$, where $k$ is the number of components of $L$. Hence $m\geq 2|\sum_{i<j}\text{lk}(L_i,L_j)|+n-k$

Therefore, we find that 
\begin{align}
\deg\nabla (L)&=m-n+1\nonumber\\
&\geq 2|\sum_{i<j}\text{lk}(L_i,L_j)|+n-k-n+1\nonumber\\
&=2\sum_{i<j}|\text{lk}(L_i,L_j)|-k+1.
\end{align}
\end{proof}

\begin{lemma}
The closure $L$ of the braid $((\sigma_1^{-1}\sigma_2^{-2}\sigma_1^{-1}\sigma_2\sigma_1^{-1}\sigma_2^{-2}\sigma_1^{-1})^2\sigma_1^{-2})^2$ is not the closure of a homogeneous braid.
\end{lemma}
\begin{proof}
Suppose $L$ were the closure of the square of a homogenous braid. It is 3-component link with pairwise linking numbers $(-10,-8,2)$. The degree of its Conway polynomial is 32 and hence by Lemma \ref{lem:degcon} we have $32\geq 2(10+8+2)-3+1=38$, which is clearly a contradiction
\end{proof}

Since every square of a homogeneous braid is also homogeneous, this shows that $L$ is not in the family shown to be real algebraic in \cite{bode:real}. This finishes the proof of Theorem \ref{thm:newnew}.







\end{document}